\documentclass{IEEEtran}
\usepackage{amsthm}
\usepackage{booktabs}
\usepackage{xcolor}
\usepackage{comment}
\usepackage{amssymb}
\usepackage{graphicx}
\usepackage{CJKutf8}

\newtheorem{theorem}{Theorem}

\theoremstyle{definition}
\newtheorem{definition}{Definition}

\newtheorem{proposition}{Proposition}

\newtheorem{lemma}{Lemma}

\newtheorem{corollary}{Corollary}

\newtheorem{algorithm}{Algorithm}
\newtheorem{remark}{Remark}

\usepackage[T1]{fontenc}
\usepackage{textcomp}
\usepackage[utf8]{inputenc}
\usepackage{stringenc}
\newcommand*{\uni}{}
\DeclareRobustCommand*{\uni}[1]{%
  \begingroup
    \StringEncodingConvert\x{%
      \pdfunescapehex{%
        00%
        \ifnum"#1<"100000 0\fi
        \ifnum"#1<"10000 0\fi
        \ifnum"#1<"1000 0\fi
        \ifnum"#1<"100 0\fi
        \ifnum"#1<"10 0\fi
        #1%
      }%
    }{utf32be}{utf8}%
    \everyeof{\noexpand}%
    \endlinechar=-1 %
  \edef\x{%
    \endgroup
    \scantokens\expandafter{%
      \expandafter\unexpanded\expandafter{\x}%
    }%
  }\x
}

\newcommand{\cjktext}[1]{\protect\begin{CJK*}{UTF8}{gbsn}#1\end{CJK*}}

\usepackage{cite}
\usepackage{amsmath}
\usepackage{amsfonts, amsmath}

\DeclareMathOperator{\proj}{proj}

\begin{document}

%
\title{Influence Spread in the Heterogeneous Multiplex Linear Threshold Model \thanks{This research has been supported in part by ONR grant N00014-19-1-2556 and  ARO grant W911NF-18-1-0325.} }
%
%
%
%
        \author{Yaofeng~Desmond~Zhong~\cjktext{(\uni{949F} \uni{8000} \uni{950B})},
        Vaibhav~Srivastava, and 
        Naomi~Ehrich~Leonard
\thanks{Y. D. Zhong and N. E. Leonard are with the Department of Mechanical and Aerospace Engineering, Princeton University, Princeton, NJ 08540 USA e-mail: (y.zhong@princeton.edu; naomi@princeton.edu). V. Srivastava is with the Department of Electrical and Computer Engineering, Michigan State University, East Lansing, MI 48824 USA (e-mail: vaibhav@egr.msu.edu).
}
}

%
%

\markboth{}%
{}
%



\maketitle

\begin{abstract}
The linear threshold model (LTM) has been used to study spread on single-layer networks defined by one inter-agent sensing modality and agents homogeneous in protocol. We define and analyze the heterogeneous multiplex LTM to study spread on multi-layer networks with each layer representing a different sensing modality and agents heterogeneous in protocol.  
Protocols are designed to distinguish signals from different layers: 
an agent becomes active if a sufficient number of its neighbors in each of any $a$ of the $m$ layers is active.
We focus on Protocol OR, when $a=1$, and Protocol AND, when $a=m$, which model agents that are most and least readily activated, respectively. We develop theory and algorithms to compute the size of the spread at steady state for any set of initially active agents and to analyze the role of distinguished sensing modalities, network structure, and heterogeneity. We show how heterogeneity manages the tension in spreading dynamics between sensitivity to inputs and robustness to disturbances.
\end{abstract}


%
\IEEEpeerreviewmaketitle

\section{Introduction}
%
%
%
%

\IEEEPARstart{T}{he} spread of an activity or innovation across a population of agents that sense or communicate over a network has critical consequences for a wide range of systems from biology to engineering. The adoption of a strategy, such as wearing a face mask during a pandemic, can spread across a social network even when there are only a few early adopters. The observation and response to a threat by one or more vigilant animals can spread through a social animal group. A robot that detects a change in the environment and takes action can spread its behavior across a networked robot team.

To predict and control spread, we present and analyze a new model that captures the realities of {\em multiple inter-agent sensing modalities} and {\em heterogeneity in responsiveness of agents to others}. We develop and prove the validity of new algorithms that provide the means to systematically determine the spreading influence of a set of agents as a function of multi-layer network structure and agent heterogeneity.

The linear threshold model (LTM), from Granovetter~\cite{granovetter1978threshold} and  Schelling~\cite{schelling1978micromotives}, describes the spread of an activity 
as discrete-time, discrete-valued state dynamics where an agent adopts or rejects an activity by comparing the fraction of its neighbors that have adopted the activity to its individual threshold. Kempe et al.~\cite{kempe2003maximizing} used the LTM with random thresholds to investigate spread of an activity over a population on a single-layer network.
Lim et al.~\cite{lim2015simple} introduced and analyzed the notion of cascade and contagion centralities in the model of~\cite{kempe2003maximizing}. The LTM on single-layer networks has also been studied in \cite{acemoglu2011diffusion, rosa2013non, garulli2015analysis, fardad2017linear, rossi2017threshold} and generalized to continuous-time, real-valued dynamics in~\cite{zhong2019cdc}. 

The single-layer network in the LTM represents a single sensing modality or a projection of multiple sensing modalities. Yet, in real-world systems, agents may distinguish the different sensing modalities, rather than project them, in ways that impact spread. For example, someone deciding whether or not to wear a mask may consider as separate signals what they see others doing in the neighborhood and what they hear over social media that others are doing. And, how they act on the signals may differ from person to person. A more readily activated person starts wearing a mask when they observe enough of the first {\em or} second group wearing a mask. A less readily activated person starts wearing a mask only when they observe enough of the first {\em and} second groups wearing a mask. 

In this paper we leverage multiplex (multi-layer) networks to model spread in a population of heterogeneous agents that interact through, and distinguish, multiple sensing modalities.  Multiplex networks have been used to study consensus dynamics \cite{gomez_diffusion_2013, trpevski_discrete-time_2014, shao_relative_2017, antonopoulos_opinion_2018, vasconcelos_consensus_2019}. Ya{\u{g}}an and Gligor~\cite{yaugan2012analysis} studied a multiplex LTM using a weighted average of activity across layers.  Other models of spread in the case of multiple sensing modalities are reviewed in~\cite{salehi_spreading_2015}, but most restrict to homogeneous agents.

In~\cite{zhong2017linear}, we first introduced the LTM on multiplex networks with homogeneous agents, where the graph for each layer is 
associated with a different sensing modality, and Protocols OR and AND  distinguish signals from different layers to model more and less readily activated agents, respectively.  We analyzed the duplex (two-layer) LTM with agents that are homogeneous in protocol and showed how to compute cascade centrality, an agent's influence on the steady-state size of the cascade.  Yang et al.~\cite{yang-multipex} studied the influence minimization problem for the homogeneous model of~\cite{zhong2017linear}.

Our contributions in the present paper are multifold. First, we define the heterogeneous multiplex LTM to analyze spreading dynamics on an arbitrary number of network layers with agents that employ protocols heterogeneously. Second, we define the heterogeneous multiplex live-edge model (LEM), which generalizes \cite{kempe2003maximizing}, and we introduce the live-edge tree to define reachability on this LEM. We prove a key result on equivalence of probabilities for the LTM and LEM. 

Third, we derive Algorithms \ref{alg} and \ref{alg:map} to compute influence spread for the heterogeneous multiplex LTM.  Algorithm \ref{alg} is provably correct and useful for small networks.  
Algorithm \ref{alg:map} maps the influence spread calculation to an inference problem in a Bayesian network and is efficient for large networks. We prove that calculating influence spread is \#P-complete.

Fourth, we derive analytical expressions for influence spread in  classes of multiplex networks. We show {\em how} ORs enhance and ANDs diminish spreading relative to  the projected network. Fifth, we investigate heterogeneity in spreading and show how it can be used to manage the tradeoff between sensitivity to a real input and robustness to a spurious signal.

Section II describes multiplex networks.  Sections III and IV introduce the heterogeneous multiplex LTM and LEM, respectively. We prove their equivalence in Section V.  Sections VI and VII present Algorithms 1 and 2. 
Section VIII presents analytical expressions for influence spread. Heterogeneity is studied in Section IX. We conclude in Section X.

\section{Multiplex Networks}
A \textit{multiplex network} $\mathcal{G}$ is a family of  $m\in \mathbb{N}$ directed weighted graphs $G_1,...,G_m$. Each graph $G_k=(V,E^k)$, $k = 1, \ldots, m$, is a {\em layer} of the multiplex network. The agent set $V=\{1,2,3,...,n\}$ is the same in all layers. The edge set of layer $k$ is $E^k \subseteq V \times V$ and can be different in different layers. Each edge $e^k_{i,j}\in E^k$, pointing from $i$ to $j$ in layer $k$, is assigned a weight $w^k_{i,j} \in \mathbb{R}^+$.
Here we adopt the ``sensing" convention for edges: edge $e^k_{i,j}$ exists if agent $i$ can sense agent $j$ in layer $k$. 
If edge $e^k_{i,j}$ exists, agent $j$ is an {\em out-neighbor} of agent $i$ in layer $k$. We denote the set of out-neighbors of $i$ in layer $k$ as $N_i^k$. We say that the weight of agent $i$'s out-neighbor $j$ in layer $k$ is the weight  $w^k_{i,j}$. We assume the weights of all out-neighbors for every agent sum up to $1$, i.e., $\sum_{j \in N_i^k} w^k_{i,j}=1$ for every agent $i$.  A {\em monoplex network} is a multiplex network with $m=1$, i.e., with only a single layer.

For undirected graphs, every edge is modeled with two opposing directed edges. For unweighted graphs, every edge $e^k_{i,j}$ can be assigned a weight $w^k_{i,j}=1/d^k_i$, where $d^k_i$ is the {\em out-degree} of node $i$ in layer $k$ and equal to the number of out-neighbors of node $i$ in layer $k$.
A \textit{projection network} of $\mathcal{G}$ is the graph $\proj(\mathcal{G})=(V,E)$ where $E=\cup_{k=1}^m E^k$.

\section{The Heterogeneous Multiplex LTM}
The linear threshold model (LTM) is described by a discrete-time dynamical system in which the state $x_i(t) \in \{0,1\}$ of each agent $i$  at iteration $t$ is inactive with $x_i(t) = 0$ or active with $x_i(t) = 1$. The LTM protocol determines how the active state spreads through the network. Our focus is on which agents will be active at steady state as a function of which agents are active initially.  We define  $\bar{x}_i = \lim_{t \rightarrow \infty} x_i(t)$.

In Section~\ref{monoLTM}, we recall the LTM 
for monoplex networks.  
In Section~\ref{multiLTM}, we generalize the LTM to multiplex networks by defining protocols for how the active state spreads when signals from different layers are distinguished. Our definition 
allows for heterogeneity among agents in protocol.

Let $S_t$ be the set of agents that are active by the end of iteration $t$. Once active, an agent remains active so that $S_{t-1} \subseteq S_t$. At $t=0$, all agents are inactive except the initially active set  $S_0$. Every agent in $S_0$ is called a {\em seed}.  The LTM protocol determines when inactive agents at iteration $t-1$ become active at iteration $t$. A steady state is reached when $S_{t-1} = S_t$.   

\subsection{Monoplex LTM} \label{monoLTM}
The LTM protocol on a monoplex network is defined as follows (e.g., \cite{kempe2003maximizing}). Each agent $i = 1, \ldots, n$ chooses a  threshold $\mu_i$ randomly and independently from a uniform distribution $U(0,1)$. An inactive agent $i$ at iteration $t-1$ becomes active at iteration $t$ if the sum of weights of its active out-neighbors at $t-1$ exceeds $\mu_i$, that is, if $\mu_{i} < \sum_{j\in N_i \cap S_{t-1}}w_{i,j}$.  For $n$ agents, steady state is reached by  $t \leq n$.
 
\subsection{Multiplex LTM} \label{multiLTM}
We introduce the LTM on a multiplex network with $m$ layers by defining a family of protocols as follows. Each agent $i$ chooses a threshold $\mu_i^k$ in each layer $k$ for $i = 1, \ldots, n$ and $k = 1, \ldots, m$. Each $\mu_i^k$ is randomly and independently drawn from the uniform distribution $U(0,1)$. In general, each agent has different neighbors in different layers. If the sum of weights of active out-neighbors of agent $i$ in layer $k$ at  $t-1$ exceeds $\mu_i^k$, that is, $\mu_i^k < \sum_{j\in N_i^k \cap S_{t-1}}w_{i,j}^k$, we say agent $i$ receives a positive input $y_i^k(t)=1$ from layer $k$ at  $t$. Otherwise, agent $i$ receives a neutral input $y_i^k(t)=0$. 

The protocols that determine whether or not an inactive agent at $t-1$ becomes active at $t$  account for the possibility that the inputs it receives at $t$ from the different layers may be conflicting. Let the average input agent $i$ receives at $t$ be $y_i(t)=\sum_{k=1}^m y_i^k(t)/m$. 

\begin{definition}[Multiplex LTM Protocol]
\label{def:multiplexLTM}
Given multiplex network $\mathcal{G}$ with seed set $S_0$,  the {\em multiplex LTM protocol for agent $i$}
is parametrized by $\delta_i \in [1/m,1]$ as follows: 
\begin{align}
  x_i(0) &= 1, \quad  \forall i \in S_0\\
  x_i(0) &= 0, \quad \forall i \notin S_0\\
  x_i(t) &=
  \begin{cases}
    1, & \text{if } y_i(t)\geq \delta_i \text{ or } x_i(t-1)=1 \\
    0, & \text{otherwise}. \\
  \end{cases}
\end{align}
We identify two protocols for the limiting values of $\delta_i$:

\noindent
{\em Protocol OR}: $\delta_i=1/m$. Inactive agent $i$ at iteration $t-1$ becomes active at iteration $t$  if it receives a positive input from {\em any} layer at $t$;

\noindent
{\em Protocol AND}: $\delta_i=1$. Inactive agent $i$ at iteration $t-1$ becomes active at iteration $t$ if it receives positive inputs from {\em all} layers at $t$. 
$\qed$
\end{definition}

The multiplex LTM protocol specifies that inactive agent $i$ at iteration $t-1$ becomes active at iteration $t$ if it receives a positive input from any $a_i\in \{1, \ldots, m \}$ of the $m$ layers, where $(a_i - 1)/m < \delta_i \leq a_i/m$.  Asymmetric sensitivity to layers can be modelled with $y_i$ a convex combination of $y_i^k$.  

In this paper, we examine the two limiting cases: Protocol OR, where $a_i=1$, and Protocol AND, where $a_i = m$. Analysis in these cases is sufficient for understanding heterogeneity and spreading dynamics on multi-layer networks. Our theory can be extended to protocols for  $a_i \in \{2, \ldots, m-1\}$. 

\begin{remark}
\label{rem:spreading}
Protocol OR models agents that are readily activated: there only needs to be sufficient activity among neighbors in one layer at $t-1$ in order for agents to become active at $t$.  Protocol AND models agents that are conservatively activated: there needs to be sufficient activity among neighbors in every layer at $t-1$ in order for agents to become active at $t$. Thus, agents with Protocol OR enhance spreading and agents with Protocol AND diminish spreading.
\end{remark}

We study heterogeneous networks in which some agents use Protocol OR while the others use Protocol AND.  
\begin{definition}[Sequence of Protocols] Let $u_i \in \{\text{OR},\text{AND}\}$ be the protocol used by agent $i$.  We define the {\em sequence of protocols} $\mathcal{U} = (u_1,u_2,...,u_n)$ to be the protocols used by the $n$ agents ordered from agent 1 to agent $n$.  
\end{definition}

\begin{lemma}
    For a multiplex network $\mathcal{G}$ with $n$ agents, the multiplex LTM converges in at most $n$ iterations.
\end{lemma}
\begin{proof}
Assume the multiplex LTM converges in more than $n$ iterations. Then at least one agent switches from inactive to active in each of the first $n$ iterations and these agents are distinct. There is at least one agent in the initial active set that is not one of those $n$ agents. This implies at least $n+1$ agents in the network, which is a contradiction. 
\end{proof}

\section{The Heterogeneous Multiplex LEM}
Our approach to analyzing the multiplex LTM generalizes the  approach in~\cite{kempe2003maximizing}, which uses the live-edge model (LEM) for monoplex networks to analyze the monoplex LTM.  In this section we define the multiplex LEM.
In Section~\ref{monoLEM}, we recall the LEM proposed in~\cite{kempe2003maximizing} for monoplex networks. In Section~\ref{duplexLEM}, we generalize the LEM to multiplex networks and introduce the notion of reachability-$\mathcal{U}$ on the multiplex LEM. Unlike in our earlier work \cite{zhong2017linear}, the reachability we propose here allows for heterogeneous protocols among agents.

\subsection{Monoplex LEM and Reachability}
\label{monoLEM}
The LEM for a monoplex network is defined as follows \cite{kempe2003maximizing}. Let $S_0$ be the set of seeds. Each unseeded agent randomly selects one of its outgoing edges with probability given by the edge weight. The selected edge is labeled as ``live", while the unselected edges are labeled as ``blocked". The seeds block all of their outgoing edges. Every directed edge will thus be either live or blocked.  The choice of edges that are live is called a \emph{selection of live edges}. 

Let $L$ be the set of all possible selections of live edges. The probability $q_l$ of  selection  $l \in L$ is the product of the weights of the live edges in selection $l$.
Because the selection of live edges can be done at the same time for every node, the LEM can be viewed as a static model. The LEM can alternatively be viewed as an iterative process in the case the live edges are selected sequentially.

A \textit{live-edge path} \cite{kempe2003maximizing} is a directed path that consists only of live edges.  Let $\mathcal{L}_{ij}$ be the set of all possible distinct live-edge paths from agent $i \notin S_0$ to $j \in S_0$.  
The probability $r_{\alpha}$ of live-edge path $\alpha\in \mathcal{L}_{ij}$ is the product of the edge weights along the path. We say $i\notin S_0$ is {\em reachable from $j\in S_0$ by live-edge path $\alpha$ with probability $r_{\alpha}$},  
and $i\notin S_0$ is {\em reachable from $j\in S_0$ with probability $r_{ij}$}, where $r_{ij} = \sum_{\alpha \in \mathcal{L}_{ij}}r_{\alpha}$.

Alternatively, we can compute $r_{ij}$ in terms of selections of live-edges. Let $L_{ij} \subseteq L$ be the set of all selections of live edges that contain a live-edge path from $i\notin S_0$ to $j\in S_0$. 
Then, $r_{ij} = \sum_{l \in L_{ij}} q_l$. 
Likewise, let $L_{iS_0} \subseteq L$ be the set of all selections of live edges that contain a live-edge path from $i\notin S_0$ to at least one node $j \in S_0$. Then, $i\notin S_0$ is {\em reachable from $S_0$ with probability $r_{iS_0}$}, where $r_{iS_0} = \sum_{l \in L_{iS_0}} q_l$.

\subsection{Multiplex LEM and Reachability}
\label{duplexLEM}

We introduce the LEM on a multiplex network as follows.
\begin{definition}[Multiplex LEM]
\label{def:multiplexLEM}
Consider a multiplex network $\mathcal{G}$ with seed set $S_0$.
In each layer $k$, each unseeded agent $i$ randomly selects one of its outgoing edges $e^k_{i,j_k}$ with probability $w^k_{i,j_k}$. The selected edges are labeled as ``live", while the unselected edges are labeled as ``blocked". The seeds block all of their outgoing edges in every layer. The choice of edges that are live is a \emph{multiplex selection of live edges}.  Let $L$ be the set of all possible multiplex selections of live edges. The probability $q_l$ of selection $l \in L$ is the product of the weights of all live edges in selection $l$. 
\end{definition}

 The challenge in generalizing the LEM to multiplex networks is in properly defining reachability. Here we introduce the  live-edge tree, which we use to define reachability. 

\begin{definition}[Live-edge Tree]\footnote{To highlight  key differences between multiplex and monoplex networks, we assume each $i \notin S_0$ has at least one neighbor in each layer. If not, 
with a slight modification of Defs.~\ref{def:live_edge_tree}- \ref{def:reachability}, the theory and computation are still valid.}
\label{def:live_edge_tree}
 Given a set of seeds $S_0$ and a multiplex selection of live edges $l \in L$, the {\em live-edge tree} $T^l_i$ associated with  agent $i \notin S_0$ is constructed as follows with agent $i$ as the root node.  
Let $e^k_{i, j_k}$ be the live edge of agent $i$ in layer $k$, $k = 1, \ldots, m$.  Then the children of the root node are agents $j_1, j_2,..., j_m$, and the root node is connected to each child with the live edge in the corresponding layer. The tree is constructed recursively in this way for each child that itself has at least one child.  Any agent in the network may appear multiple times as a node in the tree. 
\end{definition}

Fig.~\ref{fig:ex_multi3} provides an example of a three-layer multiplex network with five agents. The network has only one possible multiplex selection of live edges, given in Fig.~\ref{fig:ex_live_select}. Fig.~\ref{fig:ex_live_tree} shows the corresponding live-edge tree associated with agent 5.  

\begin{figure}
    \centering
    \includegraphics[width=2.5in]{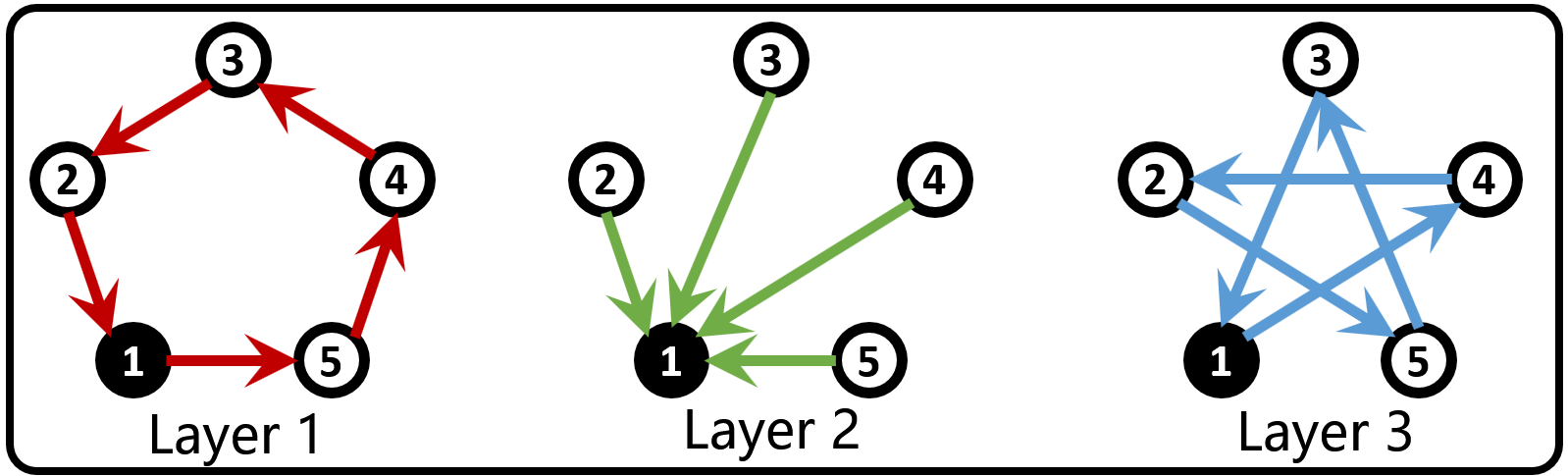}
    \caption{An example of a three-layer multiplex network with five agents. Agent 1 is the seed, which is denoted by the black circle.}
    \label{fig:ex_multi3}
\end{figure}
\begin{figure}
    \centering
    \includegraphics[width=2.5in]{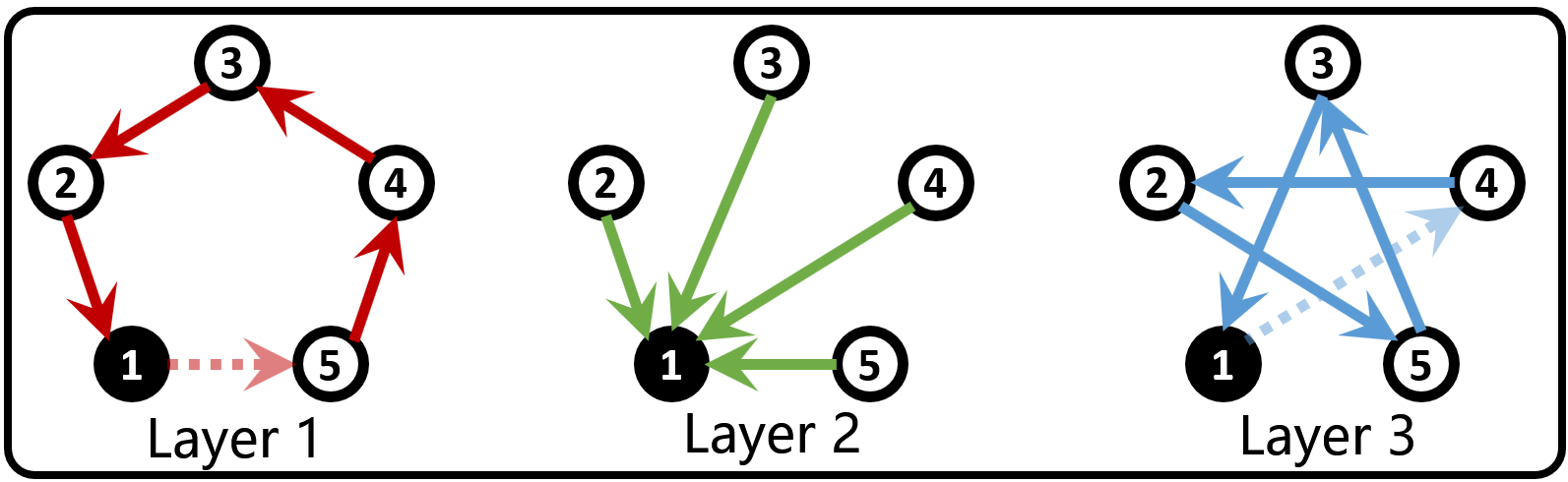}
    \caption{The unique multiplex selection of live edges for the network in Fig.~\ref{fig:ex_multi3}. 
    }
    \label{fig:ex_live_select}
\end{figure}
\begin{figure}
    \centering
    \includegraphics[width=2.1in]{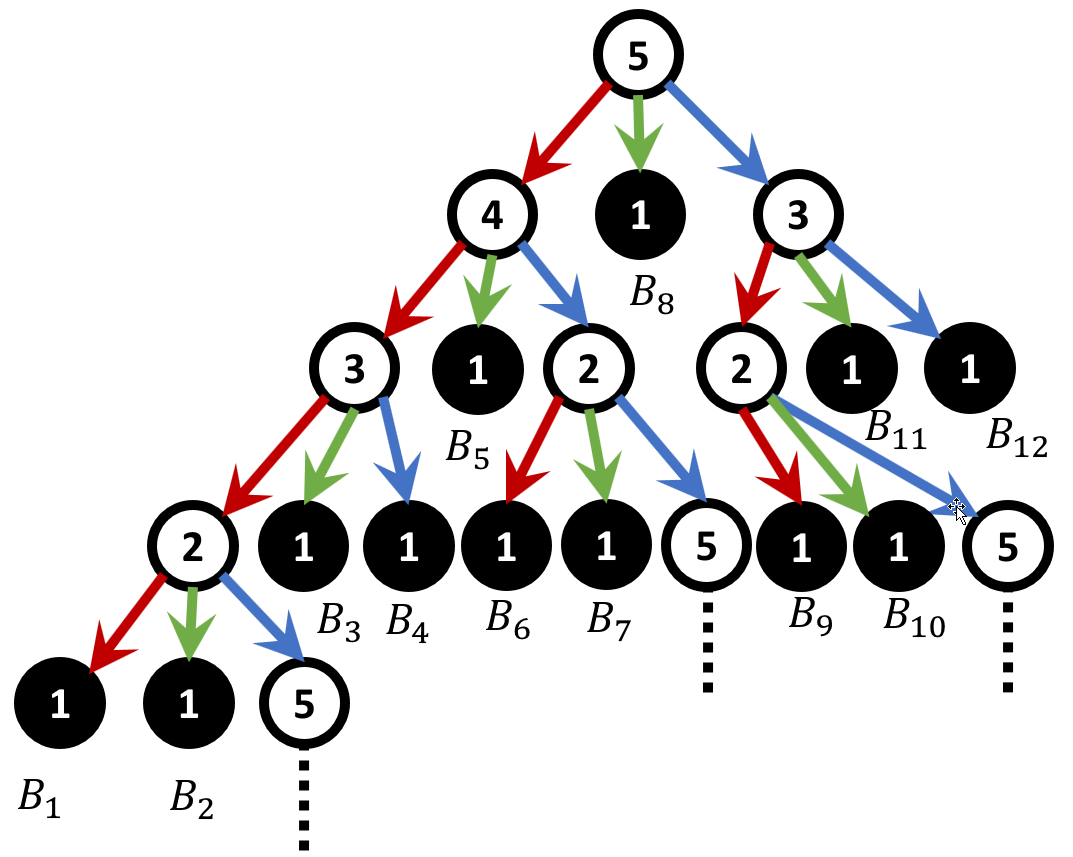}
    \caption{The live-edge tree associated with agent 5 for the example three-layer multiplex network of Fig.~\ref{fig:ex_multi3} and the unique  selection of live edges of Fig.~\ref{fig:ex_live_select}.    $\mathcal{B}_5 = \{B_1, B_2, ..., B_{12}\}$ is the set of distinct branches that end with a seed. }
    \label{fig:ex_live_tree}
\end{figure}
We next define reachability from $S_0$ of an unseeded agent in a multiplex network under a sequence of  protocols $\mathcal{U}$. 

\begin{definition}[$\mathcal{U}$-Reachability]
\label{def:reachability}
    Consider multiplex network $\mathcal{G}$ with seed set $S_0$ and multiplex selection of live edges $l \in L$. Let $T^l_{i}$ be the live-edge tree  associated with agent $i \notin S_0$. Suppose there are $b$ distinct branches in $T^l_{i}$ indexed by $\beta = 1, \ldots, b$ and of the form: $B_\beta = (i, e_{i, i_1}^{k_0}, i_1, e_{i_1,i_2}^{k_1}, i_2, ..., i_s)$,  where $i_j \in V$, $j = 1, \dots, s$, $i_s \in S_0$, and each agent in $V$ appears at most once in $B_\beta$. 
    We call each $B_\beta$ a {\em distinct branch that ends in a seed}. 
   Denote the set of these branches as $\mathcal{B}^l_{i} = \{B_1, B_2, ..., B_b\}$. For any subset $\hat{\mathcal{B}} \subseteq \mathcal{B}^l_{i}$, let the set of agents in  $\hat{\mathcal{B}}$ be $\hat{V}$ and  the set of edges in $\hat{\mathcal{B}}$ be  $\hat{E}$.

    Given a sequence of protocols $\mathcal{U}$, we say that branch subset $\hat{\mathcal{B}}\subseteq \mathcal{B}^l_{i}$ is {\em $\mathcal{U}$-feasible for $i$} if $\hat{\mathcal{B}}\neq \varnothing$ and for every $\hat{i} \in \hat{V} \setminus S_0$ for which $u_{\hat{i}} = \mathrm{AND}$, all of $\hat{i}$'s live edges belong to $\hat{E}$. Then, $i$ is {\em $\mathcal{U}$-reachable from $S_0$ by the selection of live edges $l$ with probability $q_l$} if there exists at least one $\hat{\mathcal{B}}\subseteq \mathcal{B}^l_{i}$ that is $\mathcal{U}$-feasible for $i$. Let $L^{\mathcal{U}}_{iS_0} \subseteq L$ be the set of all selections of live edges by which $i$ is $\mathcal{U}$-reachable from $S_0$.  Then,  $i$ is {\em $\mathcal{U}$-reachable from $S_0$ with probability $r^{\mathcal{U}}_{iS_0}$}, where $r^{\mathcal{U}}_{iS_0} = \sum_{l \in L^{\mathcal{U}}_{iS_0}} q_l$. $\qed$
\end{definition}



\begin{remark}
The condition for $\mathcal{U}$-feasibility for $i$ of a branch subset $\hat{\mathcal{B}} \subseteq \mathcal{B}^l_i$ does not make explicit a condition on any $\hat i \in \hat V\setminus S_0$ for which $u_{\hat{i}} = \mathrm{OR}$. This follows since for any such agent $\hat{i}$, the condition is that at least one of its live edges must be in $\hat{E}$ and this is always true by definition. 
\end{remark}


To illustrate $\mathcal{U}$-reachability, consider the live-edge tree associated with agent 5 in Fig.~\ref{fig:ex_live_tree} for the unique selection of live edges in Fig.~\ref{fig:ex_live_select} for the multiplex network of Fig.~\ref{fig:ex_multi3} with seed set $S_0 = \{ 1 \}$.    Because the selection of live edges in Fig.~\ref{fig:ex_live_select} is unique, it is chosen with probability $q=1$.  Therefore, agent $i \notin S_0$ is $\mathcal{U}$-reachable from $S_0$ with probability 1 if there exists at least one $\hat{\mathcal{B}} \subseteq  \mathcal{B}_i$ that is $\mathcal{U}$-feasible for $i$.
For agent 5, there are 12 distinct branches that end in a seed, as shown in Fig.~\ref{fig:ex_live_tree}; thus, $\mathcal{B}_5 = \{B_1, B_2, ..., B_{12}\}$. For example, $B_8 = (5, e_{5,1}^2, 1)$.

We compute $\mathcal{U}$-reachability from $S_0$ for agent 5 for each the following three sequences of protocols used by the five agents in the three-layer multiplex network:
\begin{align}
    \mathcal{U}_1 &= (\text{OR},\text{AND},\text{AND},\text{AND},\text{OR}) \\
    \mathcal{U}_2 &= (\text{OR},\text{OR},\text{AND},\text{AND},\text{AND}) \\
    \mathcal{U}_3 &= (\text{OR},\text{AND},\text{AND},\text{AND},\text{AND}).
\end{align}
\begin{enumerate}
    \item Let $\mathcal{U} = \mathcal{U}_1$.  Consider $\hat{\mathcal{B}}=\{B_8\} \subset \mathcal{B}_5$. Then, $\hat{V} = \{1,5\}$ and $\hat{E} = \{e^2_{5,1}\}$.  Since 5 is the only unseeded node in $\hat{V}$ and $u_5 = \mathrm{OR}$, $\hat{\mathcal{B}}$ is $\mathcal{U}$-feasible for 5.  Thus, agent 5 is $\mathcal{U}$-reachable from $S_0$ with probability 1.
    
    \item Let $\mathcal{U} = \mathcal{U}_2$.  Consider $\hat{\mathcal{B}}=\mathcal{B}_5$.  Then, $\hat{V} = \{1,2,3,4,5\}$.  The unseeded nodes $j \in \hat{V}$ for which $u_j = \mathrm{AND}$ are $j = 3,4,5$. From Fig.~\ref{fig:ex_live_tree}, observe that all the live edges of nodes 3, 4, and 5, belong to $\hat{E}$, the edge set of $\hat{\mathcal{B}}=\mathcal{B}_5$.  Thus, $\hat{\mathcal{B}}$ is $\mathcal{U}$-feasible for 5, and agent 5 is $\mathcal{U}$-reachable from $S_0$ with probability 1.
    \item Let $\mathcal{U} = \mathcal{U}_3$. In this case there is no $\mathcal{U}$-feasible subset $\hat{\mathcal{B}} \subseteq \mathcal{B}_5$, since $u_2 = \mathrm{AND}$ and agent 2 has a live edge $e^3_{2,5}$, which is not in the edge set of any branch in $\mathcal{B}_5$.  Thus, agent 5 is not $\mathcal{U}$-reachable from $S_0$.
    
\end{enumerate}

In the next section we prove the equivalence of the probability that agent $i$ is $\mathcal{U}$-reachable from $S_0$ for the multiplex LEM and the probability that agent $i$ is active at steady state for the multiplex LTM with seed set $S_0$. For our example, this implies that under $\mathcal{U}_1$ or $\mathcal{U}_2$, agent 5 will become active at steady state with probability 1 and under  $\mathcal{U}_3$, agent 5 will remain inactive at steady state. 

\section{Equivalence of LTM and LEM}

The monoplex LEM was introduced in \cite{kempe2003maximizing} and proved to be equivalent to the monoplex LTM in the sense that the probabilities of agents being reachable from a set $S_0$ in the LEM are equal to the probabilities of agents being active at steady state given seed set $S_0$ in the LTM.   Computing these probabilities for the LTM is challenging because it requires solving over temporal iterations. However, leveraging the equivalence, the probability distributions can be computed without temporal iteration using the LEM as a static model. 




The equivalence is recalled in Section~\ref{sec:monoplex_equiv} for monoplex networks and proved in Section~\ref{sec:multiplex_equiv} for multiplex networks.

\subsection{Equivalence for Monoplex Networks}
\label{sec:monoplex_equiv}

The LTM and LEM were proved to be equivalent in   \cite{kempe2003maximizing} in the following sense.
    For a given monoplex network $G$ with seed set $S_0$, the probabilities of the following two events for arbitrary agent $i \notin S_0$ are the same: 
    \begin{enumerate}
        \item{$i$ is active at steady state 
        for the LTM with random thresholds and initial active set $S_0$;} 
        \item{$i$ is reachable from set $S_0$ under the random selection of live edges in the LEM.} 
    \end{enumerate}


\subsection{Equivalence for Multiplex Networks}
\label{sec:multiplex_equiv}
We generalize the equivalence of LTM and LEM to multiplex networks in this section. First, we prove the following lemma that infers an agent's $\mathcal{U}$-reachability from the $\mathcal{U}$-reachability of its children in the live-edge tree. We then leverage this lemma to prove the equivalence in Theorem~\ref{Pequivalent}. 
\begin{lemma}
    \label{lemma_0}
    Given a multiplex network $\mathcal{G}$ with seed set $S_0$, multiplex selection of live edges $l\in L$ and sequence of  protocols $\mathcal{U}$, consider agent $i\notin S_0$ and its associated live-edge tree $T^l_{i}$. Assume $i$'s live edge in layer $k$ connects to agent $i_1^k$, $k= 1, \ldots, m$. Then the $\mathcal{U}$-reachability of $i$ from $S_0$ by  selection $l$ can be inferred from the reachability of its children $i_1^k$ and its protocol $u_{i}$ as follows:
    \begin{enumerate}
        \item Let $u_{i}=\text{OR}$. Then, $i$ is $\mathcal{U}$-reachable from $S_0$ by selection of live edges $l$ if and only if at least one child $i_1^k$ is $\mathcal{U}$-reachable from $S_0$  by selection of live edges $l$.
        \item Let $u_{i}=\text{AND}$.  Then, $i$ is $\mathcal{U}$-reachable from $S_0$ by selection of live edges $l$ if and only if every child $i_1^k$ is $\mathcal{U}$-reachable from $S_0$  by selection of live edges $l$.
  \end{enumerate}
\end{lemma}

\begin{proof}
    Let $T^l_{i_1^k}$ be the corresponding live-edge tree associated with agent $i$'s child $i_1^k$ for $k=1, \ldots, m$. 
    
    1) Let $u_{i}=\text{OR}$ and suppose $i_1^k$ is $\mathcal{U}$-reachable from $S_0$ by selection $l$ for some $k$. By Definition~\ref{def:reachability}, the set $\mathcal{B}^l_{i_1^k}$ of distinct branches that end with a seed in $T^l_{i_1^k}$ is nonempty and there exists a  subset $\hat{\mathcal{B}}_k \subseteq \mathcal{B}^l_{i_1^k}$ that is $\mathcal{U}$-feasible for $i_1^k$ such that for every $\hat{i} \in \hat{V}_k \setminus S_0$ for which $u_{\hat{i}} = \mathrm{AND}$, all of $\hat{i}$'s live edges belong to $\hat{E}_k$, where  $\hat{V}_k$ and $\hat{E}_k$ are the sets of agents and edges in $\hat{\mathcal{B}}_k$, respectively.

    For every branch $B^k_r=(i_1^k, e_{i_1^k,i_2}^{k'}, i_{2}, ..., i_{s}) \in\hat{\mathcal{B}}_k$, there exists a branch $B^{k0}_r= (i, e_{i,i_1^k}^k, i_1^k, e_{i_1^k,i_2}^{k'}, i_2,..., i_s)$ in $T^l_{i}$. Let $\hat{\mathcal{B}}_{k0} \subseteq \mathcal{B}^l_i$ be the set of all branches in $T^l_{i}$ that correspond to branches in $\hat{\mathcal{B}}_k$. Then, $\hat{V}_{k0}=\hat{V}_k \cup \{i\}$ is the set of agents in $\hat{\mathcal{B}}_{k0}$ and $\hat{E}_{k0}=\hat{E}_k \cup \{e_{i,i_1^k}^k\}$ is the set of edges in $\hat{\mathcal{B}}_{k0}$. Thus, since $u_i = \mathrm{OR}$ and $\hat{\mathcal{B}}_k$ is $\mathcal{U}$-feasible for $i_1^k$, by Definition~\ref{def:reachability}, $\hat{\mathcal{B}}_{k0}$ must  be $\mathcal{U}$-feasible for $i$. Agent $i$ is therefore $\mathcal{U}$-reachable from $S_0$  by selection $l$.  This proves the ``if'' part of the statement.
    

    If no child $i_1^k$ is $\mathcal{U}$-reachable from $S_0$ by selection $l$, then there exists no nonempty $\hat{\mathcal{B}}_k$ in $T^l_{i_k}$ that is $\mathcal{U}$-feasible for $i^k_1$. This implies there is no nonempty $\hat{\mathcal{B}}_{k0}$ in $T^l_{i}$ that is $\mathcal{U}$-feasible for $i$. Therefore, agent $i$ cannot be $\mathcal{U}$-reachable from $S_0$  by selection $l$. This proves the ``only if'' part of the statement.
    
    2) Let $u_{i}=\text{AND}$ and suppose $i_1^k$ is $\mathcal{U}$-reachable from $S_0$ by selection $l$, for every $k=1, \ldots, m$.  By Definition~\ref{def:reachability}, for every $k$, there exists a  subset $\hat{\mathcal{B}}_k \subseteq \mathcal{B}^l_{i_1^k}$ in $T^l_{i_1^k}$ that is $\mathcal{U}$-feasible for $i_1^k$.  For every $k$ let $\hat{\mathcal{B}}_{k0} \subseteq \mathcal{B}^l_i$ be the set of all branches in $T^l_{i}$ that correspond to branches in $\hat{\mathcal{B}}_k$ as defined in the proof of 1).  Let $\hat{\mathcal{B}}_0 = \cup_{k=1}^m \hat{\mathcal{B}}_{k0} \subseteq \mathcal{B}^l_i$ with agent set $\hat{V}_0$ and edge set $\hat{E}_0$.  By construction, $i \in \hat{V}_0$ and all of agent $i$'s live edges belong to $\hat{E}_0$.  It follows  that $\hat{\mathcal{B}}_0$ is $\mathcal{U}$-feasible for $i$ and thus agent $i$ is $\mathcal{U}$-reachable from $S_0$ by selection $l$.  This proves the ``if'' part of the statement. 
    

If there is one child $i_1^k$ that is not $\mathcal{U}$-reachable from $S_0$ by selection $l$, then there exists no set $\hat{\mathcal{B}}_k$ in $T^l_{i_k}$ that is $\mathcal{U}$-feasible for $i^k_1$.
Suppose there is a set $\hat{\mathcal{B}}_0\subseteq \mathcal{B}^l_i$ in $T^l_{i}$ that is $\mathcal{U}$-feasible for $i$. Since $u_i=\mathrm{AND}$, by Definition~\ref{def:reachability}, it follows that edge $e_{i,i_1^k}^k \in \hat{E}_0$. For every branch in $\hat{\mathcal{B}}_0$ that starts with $i$ and edge $e_{i,i_1^k}^k$, i.e., $B^{k0}_r= (i, e_{i,i_1^k}^k, i_1^k, e_{i_1^k,i_2}^{k'} i_2,..., i_s)$, we denote the set of all corresponding branches $B^k_r=(i_1^k, e_{i_1,i_2}^{k'}, i_{2}, ..., i_{s})$ as $\hat{\mathcal{B}}_k \subseteq \mathcal{B}^l_{i_1^k}$. It follows that  $\hat{\mathcal{B}}_k$ is  $\mathcal{U}$-feasible  for $i_1^k$ in $T^l_{i_1^k}$ and $i_1^k$ is reachable from $S_0$.  This is a contradiction. Thus, there is no set $\hat{\mathcal{B}}_0$ that is $\mathcal{U}$-feasible for $i$ and agent $i$ cannot be $\mathcal{U}$-reachable from $S_0$ by selection $l$. This proves the ``only if'' part of the statement. 
\end{proof}

While Definition \ref{def:reachability} uses the LEM to define $\mathcal{U}$-reachability in a static way, the LEM can also be used to reveal the $\mathcal{U}$-reachability of agents as an iterative process over time \cite{kempe2003maximizing} 
as follows. 
First, using Lemma~\ref{lemma_0}, 
determine the $\mathcal{U}$-reachability of the agents with at least one edge coming from initial set $S_0' = S_0$. 
If an agent is determined to be $\mathcal{U}$-reachable from $S_0'$, add it to $S_0'$ to get a new reachable set $S_1'$. In the next iteration, follow the same procedure and get a sequence of reachable sets $S_0', S_1', S_2',...$. The process ends at iteration $t$ if $S_t'=S_{t-1}'$, where $S_t'$ is the set of agents that are $\mathcal{U}$-reachable from $S_0$. The mapping between static and temporal determinations of reachability for the LEM is the key to proving the equivalence of the multiplex LTM and multiplex LEM.  We also make use of the following definitions. 

\begin{definition}[LTM-related events]
    We define $f^k_i(t)$ ($g^k_i(t)$) to be the event that the sum of weights of active out-neighbors of $i$ in layer $k$ does (does not) exceed $\mu_i^k$ at $t$:  
    $f^k_i(t) = \{\mu_i^k<\sum_{j\in N_i^k \cap S_{t}}w^k_{i,j}\}$ and $g^k_i(t)=\{\mu_i^k\geq \sum_{j\in N_i^k \cap S_{t}}w^k_{i,j}\}$. Let
    \[ 
    X_k = f^k_i(t), \quad Y_k= g^k_i(t-1), \quad Z_k = f^k_i(t-1).
    \] 
\end{definition}
\begin{definition}[LTM-related probabilities]
    For agent $i$ that is inactive at $t$,  we define the probability that $i$ becomes active at $t+1$ as $\mathbb{P}_{i(\mathrm{OR})}^{t+1}$ if $i$ uses Protocol OR, and as $\mathbb{P}_{i(\mathrm{AND})}^{t+1}$ if $i$ uses Protocol AND.
\end{definition}

\begin{definition}[LEM-related events]
    We define $f'^k_i(t)$ ($g'^k_i(t)$) to be the event that agent $i$'s live edge in layer $k$ does (does not) connect to the reachable set $S_t'$ at $t$. Let 
    \[
    X_k'=f'^k_i(t), \quad Y_k'=g'^k_i(t-1), \quad Z_k'=f'^k_i(t-1).
    \]
\end{definition}
\begin{definition}[LEM-related probabilities]
    Consider the LEM as an iterative process. If agent $i\notin S_t'$, then we define the probability that $i\in S_{t+1}'$ as $\mathbb{P}_{i(\mathrm{OR})}^{'t+1}$ if $i$ uses Protocol OR, and as $\mathbb{P}_{i(\mathrm{AND})}^{'t+1}$ if $i$ uses Protocol AND.
\end{definition}
We state the equivalence of multiplex LTM and multiplex LEM in the following theorem.
\begin{theorem}
    \label{Pequivalent}
    For a multiplex network $\mathcal{G}$ with seed set $S_0$,  multiplex selection of live edges $l\in L$ and sequence of  protocols $\mathcal{U}$,
    the probabilities of the following two events regarding an arbitrary agent $i \notin S_0$ are the same:
    \begin{enumerate}
    \item $i$ is active at steady state 
    for the multiplex LTM under $\mathcal{U}$ with random thresholds and initial active set $S_0$;
    \item $i$ is $\mathcal{U}$-reachable from the set $S_0$ under random selection of live edges in the multiplex LEM.
    \end{enumerate}
\end{theorem}
\begin{proof}
    We prove by mathematical induction.
    
    We use $\mathbb{P}(X_k|Y_k)=\mathbb{P}(X_k'|Y_k')$, which we have from~\cite{kempe2003maximizing}.
    We show i) $\mathbb{P}_{i(\mathrm{OR})}^{t+1} = \mathbb{P}_{i(\mathrm{OR})}^{'t+1}$ and ii) $\mathbb{P}_{i(\mathrm{AND})}^{t+1} = \mathbb{P}_{i(\mathrm{AND})}^{'t+1}$, so  by induction over the iterations, the probabilities of the two events in the statement of the theorem are the same.
    
    \noindent
    i) \textit{Proving} $\mathbb{P}_{i(\mathrm{OR})}^{t+1} = \mathbb{P}_{i(\mathrm{OR})}^{'t+1}$ \textit{(when $i$ uses Protocol OR)}

    In the LTM, $i$ being inactive at $t$ means none of $i$'s thresholds is exceeded at $t-1$ and $i$  being active at $t+1$ means $i$'s thresholds in at least one layer is exceeded at $t$. In this case, the probability that $\mu_i^k$ is exceeded at $t$  is $\mathbb{P}_k=\mathbb{P}(X_k|Y_1,Y_2, ..., Y_m)=\mathbb{P}(X_k|Y_k)$. The last equality holds because random variables $\mu_i^1$, $\mu_i^2$, ..., $\mu_i^m$ are independent. Then $\mathbb{P}_{i(\mathrm{OR})}^{t+1}$ is the complement of the probability that its threshold in none of the layers are exceeded, i.e., $\mathbb{P}_{i(\mathrm{OR})}^{t+1}= 1-\prod_{k=1}^m(1-\mathbb{P}_k)$ 

    In the LEM, by Lemma \ref{lemma_0}, $i\notin S_t'$ means all of $i$'s live edges are not connected to $S_{t-1}'$ and $i\in S_{t+1}'$ means at least one live edge of $i$ is connected to $S_t'$. In this case, the probability that $i$'s live edge in layer $k$ connects to $S_t'$ is $\mathbb{P}_k'=\mathbb{P}(X_k'|Y_1',Y_2',...,Y_k')=\mathbb{P}(X_k'|Y_k')$. Then $\mathbb{P}_{i(\mathrm{OR})}^{'t+1}$ is the complement of the probability that none of $i$'s live edges connects to $S_t'$, i.e., $\mathbb{P}_{i(\mathrm{OR})}^{'t+1}=1-\prod_{k=1}^m(1-\mathbb{P}_k')$.
    
    From $\mathbb{P}(X_k|Y_k)=\mathbb{P}(X_k'|Y_k')$ we have $\mathbb{P}_k=\mathbb{P}_k', k=1,2,...,m$. So that we conclude that $\mathbb{P}_{i(\mathrm{OR})}^{t+1}=\mathbb{P}_{i(\mathrm{OR})}^{'t+1}$.
    
    \noindent
    ii) \textit{Proving} $\mathbb{P}_{i(\mathrm{AND})}^{t+1} = \mathbb{P}_{i(\mathrm{AND})}^{'t+1}$ \textit{(when $i$ uses Protocol AND)}
    
    In the LTM, $i$ being inactive at $t$ means at least one of its thresholds is not exceeded at $t-1$ and being active at $t+1$ means all of $i$'s thresholds are exceeded at $t$. In this case, there are $2^m-1$ possible events, with probabilities denoted as $\mathbb{P}_{m+1}, \ldots, \mathbb{P}_{2^m+m-1}$. We have that $\mathbb{P}_{m+1}=\mathbb{P}(X_1,X_2,...,X_m|Y_1,Y_2,...,Y_m)$. The other probabilities have a similar form but with one or more, but not all, of the $Y_k$ replaced by $Z_k$. Since $Y_k$ and $Z_k$ are mutually exclusive for all $k$, we have $\mathbb{P}_{i(\mathrm{AND})}^{t+1}=\sum_{l=m+1}^{2^m+m-1} \mathbb{P}_l$.

    In the LEM, by Lemma \ref{lemma_0}, $i\notin S_t'$ means at least one of $i$'s live edges are not connected to $S_{t-1}'$ and $i\in S_{t+1}'$ means all of $i$'s live edges are connected to $S_t'$. In this case, there are $2^m-1$ possible events, with probabilities denoted as $\mathbb{P}_{m+1}', \ldots, \mathbb{P}_{2^m+m-1}'$. We have that $\mathbb{P}_{m+1}'=\mathbb{P}(X_1',X_2', ..., X_m'|Y_1',Y_2', ..., Y_m')$. The other probabilities have similar form but with one or more, but not all, of the $Y_k'$ replaced by $Z_k'$. Since $Y_k'$ and $Z_k'$ are mutual exclusive for all $k$, we have $\mathbb{P}_{i(\mathrm{AND})}^{'t+1}=\sum_{l=m+1}^{2^m+m-1} \mathbb{P}_l'$.

    Since the thresholds are independent of one another, $\mathbb{P}_l$ can be separated into the product of $m$ terms, each of which involves events associated to one layer only: $\mathbb{P}_l=\prod_{k=1}^{m}\mathbb{P}_{0k}$, where $\mathbb{P}_{0k}$ is either $\mathbb{P}(X_k,Y_k)/\mathbb{P}(Y_k)=\mathbb{P}(X_k|Y_k)$ or $\mathbb{P}(X_k,Z_k)/\mathbb{P}(Z_k)=1$. Since the live edges that each agent uses are independent of one another, similar analysis applies to $\mathbb{P}_l'$: $\mathbb{P}_l'=\prod_{k=1}^{m}\mathbb{P}_{0k}'$, where $\mathbb{P}_{0k}'$ is either $\mathbb{P}(X_k',Y_k')/\mathbb{P}(Y_k')=\mathbb{P}(X_k'|Y_k')$ or $\mathbb{P}(X_k',Z_k')/\mathbb{P}(Z_k')=1$. It follows from $\mathbb{P}(X_k|Y_k)=\mathbb{P}(X_k'|Y_k')$ that $\mathbb{P}_l=\mathbb{P}_l', l=m+1,m+2,...,2^m+m-1$ and $\mathbb{P}_{i(\mathrm{AND})}^{t+1}=\mathbb{P}_{i(\mathrm{AND})}^{'t+1}$.
\end{proof}

Theorem \ref{Pequivalent} generalizes the equivalence of LTM and LEM to multiplex networks. Leveraging Theorem \ref{Pequivalent}, we can calculate an agent's influence, in terms of spreading information through the network, without needing to simulate the multiplex LTM.  




\section{Computing Multiplex Influence Spread}
In this section we define multiplex influence spread and multiplex cascade centrality for the LTM.  We then derive and prove the validity of an algorithm to compute them.



\subsection{Monoplex Influence Spread and Cascade Centrality} \label{sec:monocc}
The {\em monoplex influence spread of agents in $S_0$}, denoted $\sigma^{G}_{S_0}$,  is defined as the expected number of active agents at steady state for the monoplex LTM given the  network $G$ and  initial active set $S_0$~\cite{kempe2003maximizing}.
The {\em monoplex cascade centrality of agent $j$}, denoted $\mathcal{C}_j^{G}$, is the influence spread of agent $j$ defined  in~\cite{lim2015simple} as
$\mathcal{C}_j^{G} = \sigma^{G}_{j}$,
the expected number of active agents at steady state for the monoplex LTM given $G$ and $S_0 = \{j\}$. 


\subsection{Multiplex Influence Spread and Cascade Centrality} \label{sec:multicc}

Influence spread and cascade centrality are naturally generalized to the multiplex setting as  follows.


\begin{definition}[Multiplex influence spread] \label{def:mis}
The {\em multiplex influence spread of agents in $S_0$}, denoted $\sigma^{\mathcal{G},\mathcal{U}}_{ S_0}$, is defined as the expected number of active agents at steady state for the multiplex LTM given the network $\mathcal{G}$, sequence of protocols $\mathcal{U}$, and initial active set $S_0$.  Let $\mathbb{E}^{\mathcal{G},\mathcal{U}}_{S_0}$ and $\mathbb{P}^{\mathcal{G},\mathcal{U}}_{S_0}$ be  expected value and probability, respectively, conditioned on $\mathcal{G},\mathcal{U},S_0$.  Then
\begin{equation}
    \label{influenceSpread}
    \sigma^{\mathcal{G},\mathcal{U}}_{ S_0}  = \mathbb{E}^{\mathcal{G},\mathcal{U}}_{S_0}\Big(\sum_{i=1}^n \bar x_i  \Big) 
    = \sum_{i=1}^n\mathbb{P}^{\mathcal{G},\mathcal{U}}_{S_0}(\bar x_i=1). 
\end{equation}
\end{definition}

\begin{definition}[Multiplex cascade centrality] \label{def:mcc}
    The {\em multiplex cascade centrality of agent $j$}, denoted $\mathcal{C}_j^{\mathcal{\mathcal{G},U}}$, is defined as
    \begin{equation}
    \label{CascadeCen}
    \mathcal{C}_j^{\mathcal{\mathcal{G},U}} = \sigma^{\mathcal{G},\mathcal{U}}_ {j}. 
\end{equation}
\end{definition}

When $\mathcal{U}  = (u, ..., u)$, we  replace  $\mathcal{U}$ with $u$ in the superscript.   For example, when $u = \text{OR}$, we write $\mathbb{P}^{\mathcal{G},\mathrm{OR}}_{S_0}(\bar x_i=1)$  and $\mathcal{C}_j^{\mathcal{G},\mathrm{OR}}$.  When $\mathcal{G}$ is understood we drop it from the superscript.

\subsection{Computing Multiplex Influence Spread and  Centrality}

We can directly compute multiplex influence spread and multiplex cascade centrality by computing probabilities of $\mathcal{U}$-reachability for the LEM, which is much easier than computing probabilities of agents being active at steady state for the LTM. We summarize in a corollary to Theorem~\ref{Pequivalent}.
\begin{corollary}
    Given multiplex network $\mathcal{G}$ and sequence of protocols $\mathcal{U}$, multiplex influence spread of agents in $S_0$ and multiplex cascade centrality of agent $j$ can be determined as
\begin{equation}
{\sigma}_{S_0}^{\mathcal{U}} = \sum_{i=1}^n r^{\mathcal{U}}_{iS_0}, \;\;\; 
\mathcal{C}_j^{\mathcal{U}} =    \sum_{i=1}^n r^{\mathcal{U}}_{ij}.
\label{corresult}
\end{equation}
\end{corollary}
\begin{proof}
    By Definition~\ref{def:reachability},   $r_{iS_0}^{\mathcal{U}}$ is the probability that $i$ is $\mathcal{U}$-reachable from $S_0$ in the multiplex LEM of Definition~\ref{def:multiplexLEM}.  By Theorem~\ref{Pequivalent}, $r_{iS_0}^{\mathcal{U}} = \mathbb{P}^{\mathcal{U}}_{S_0}(\bar x_i=1)$. Thus, by Definition~\ref{def:mis}, $\sum_{i=1}^n r_{iS_0}^{\mathcal{U}} = \sum_{i=1}^n \mathbb{P}^{\mathcal{U}}_{S_0}(\bar x_i=1) = \sigma_{S_0}^{\mathcal{U}}$.  In case $S_0 = \{j\}$, by Definition~\ref{def:mcc}, $\mathcal{C}_j^{\mathcal{U}} = \sigma_j^{\mathcal{U}} = \sum_{i=1}^n r^{\mathcal{U}}_{ij}$.
\end{proof}

Algorithm~\ref{alg} uses \eqref{corresult} to compute multiplex influence spread and can be specialized to compute cascade centrality.
\begin{algorithm}[Compute multiplex influence spread $\sigma^{\mathcal{G},\mathcal{U}}_{ S_0}$]
    \label{alg} Given multiplex network $\mathcal{G}$ and sequence of  protocols $\mathcal{U}$: 
    \begin{enumerate}
    \item Find the set $L$ of all possible selections of live edges for multiplex network $\mathcal{G}$ and initially active set $S_0$. Calculate the probability $q_l$ of each  $l \in L$.
    \item For each agent $i$ find   $L_{iS_0} ^{\mathcal{U}} \subseteq L$, the set of all $l\in L$ such that 
   $i$ is $\mathcal{U}$-reachable  from $S_0$ by selection $l$. 
    \item Calculate 
    $r^{\mathcal{U}}_{iS_0} = \sum_{l \in L^{\mathcal{U}}_{iS_0}} q_l $.
    \item Calculate 
    $\sigma_{S_0}^{\mathcal{\mathcal{G},U}} = \sum_{i=1}^n r^{\mathcal{U}}_{iS_0}$.
    \end{enumerate}
\end{algorithm}

%


Although the algorithm is not efficient, we can use it to accurately calculate multiplex influence spread for multiplex networks with a small number of agents. Next, we propose an efficient approach that sacrifices accuracy to calculate multiplex influence spread for large networks.

\section{A Bayesian Network Approach}
In this section, we map the problem of computing multiplex influence spread into a problem of probabilistic inference in Bayesian networks (BN). This means we can compute multiplex influence spread by using an appropriate algorithm for inference in BNs, such as the loopy belief propagation algorithm. 
We first recall the definition of a BN. 
\begin{definition}[Bayesian network]
    Let $G=(V,E)$, where $V=1, 2, ..., n$ and $E\subset V \times V$, be a directed acyclic graph (DAG). Each node $i \in V$ is associated with a random variable $x'_i \in \mathcal{X}'_i$.
    Denote the set of out-neighbors of  $i \in V$ as $N'_i$. Let $\mathbb{P}(x'_i | x'_{N'_i})$ be the probability of $x'_i$ conditioned on the states of nodes in $N'_i$. Then $G$ is a {\em Bayesian network} 
    if the joint distribution of the random variables is factorized into conditional probabilities:
      $\mathbb{P}(x'_1,x'_2,...,x'_n)=\prod_{i=1}^n \mathbb{P}(x'_i|x'_{N'_i})$.
\end{definition}

Probabilistic inference in a Bayesian network refers to calculating the marginal probability of the state of each unobserved node, conditioned on the states of the observed nodes. The belief propagation (BP) algorithm was first proposed by Pearl~\cite{pearl1988probabilistic} to solve probabilistic inference in Bayesian networks. Pearl~\cite{pearl1988probabilistic} showed that the algorithm is exact on DAGs without loops, i.e., trees and polytrees~\cite{pearl1988probabilistic}. It is not guaranteed to converge when applied to DAGs with loops. However, Murphy et al.~\cite{murphy1999loopy} showed that loopy belief propagation (LBP) - the application of Pearl's algorithm to DAGs with loops - provides a good approximation when it converges. The junction tree algorithm~\cite{lauritzen1988local} was proposed to perform exact inference on a general graph. The general graph is first modified with additional edges to make it a junction tree, and then belief propagation is performed on the modified network.

In the following algorithm, we show how the joint distribution of the probability of activation of agents in a class of LTM admits the graphical structure of Bayesian network.

\begin{algorithm}[Bayesian network from multiplex LTM] 
Given multiplex network $\mathcal{G}$ for which $\proj(\mathcal{G})$ is a DAG and sequence of protocols $\mathcal{U}$:
\label{alg:map}
    \begin{enumerate}
        \item Let $G = \proj(\mathcal{G})$ be the underlying DAG for the Bayesian network. Then $N_i' = N_i=\cup_{k=1}^m N_i^k$.
        \item Let the random variable $x'_i$ of node $i$ in the Bayesian network be  $\bar{x}_i$, the steady-state value of agent $i$ for the multiplex LTM on $\mathcal{G}$. Then $x'_i \in \mathcal{X}'_i=\{0, 1\}$. 
        \item Construct the conditional probabilities for the Bayesian network in terms of the conditional probabilities for the multiplex LTM: $\mathbb{P}(x'_i | x'_{N_i}) = \mathbb{P}^{\mathcal{G}, u_i}(\bar{x}_i| \bar{x}_{N_i})$. 
    \end{enumerate}
\end{algorithm}
Since all random variables are discrete, the conditional probability $\mathbb{P}(x'_i | x'_{N_i}) = \mathbb{P}^{\mathcal{G}, u_i}(\bar{x}_i| \bar{x}_{N_i})$ can be fully described with a conditional probability table (CPT).
We show how to construct a CPT for $\mathbb{P}^{\mathcal{G}, u_i}(\bar{x}_i| \bar{x}_{N_i})$ for the Fig.~\ref{CPT} example. The CPT of $i$ has $2^{|N_i|}$ rows.
For agent 6, $N_6=\{3,4,5\}$, $\bar{x}_{N_6}=\{\bar{x}_3, \bar{x}_4, \bar{x}_5\}$ and its CPT has $2^{|N_6|}=8$ rows. The CPT provides $\mathbb{P}^{\mathcal{G}, u_6}(\bar{x}_6=0 | \bar{x}_3, \bar{x}_4, \bar{x}_5)$ and $\mathbb{P}^{\mathcal{G}, u_6}(\bar{x}_6=1| \bar{x}_3, \bar{x}_4, \bar{x}_5)$.
Table~\ref{cptOR} and Table~\ref{cptAND} are the CPTs for agent 6 when it uses Protocol OR and Protocol AND, respectively. 

\begin{figure}[htbp]
  \centering
  \includegraphics[width=1.0in]{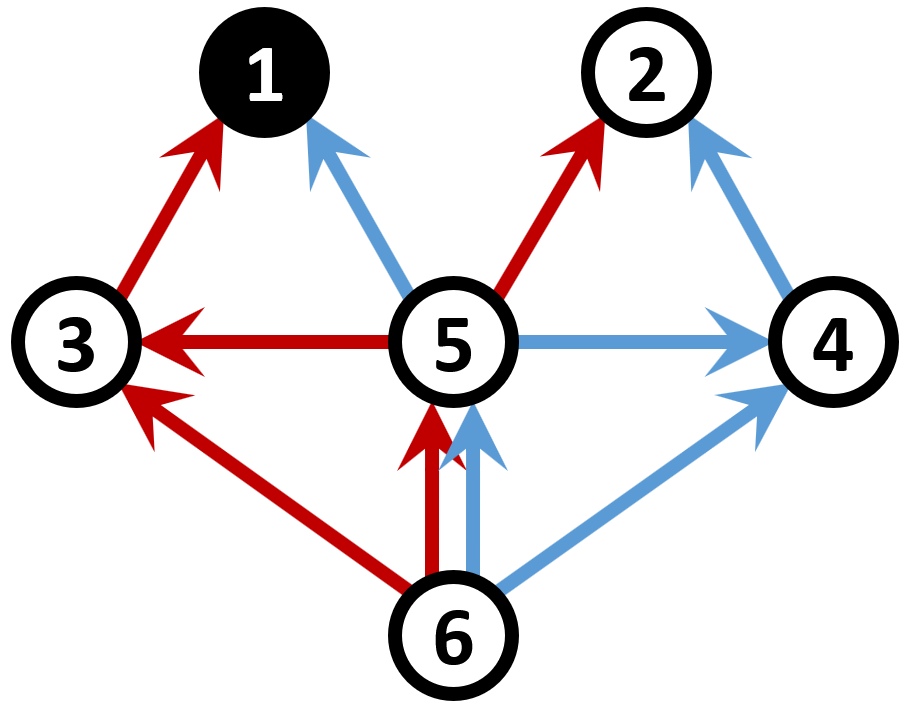}
  \caption{Multiplex network $\mathcal{G}_s$ with two unweighted layers and six agents. Red (blue) arrows represent edges in layer 1 (layer 2). 
    $S_0=\{1\}$.
  }
  \label{CPT}
\end{figure}

\begin{table}[htbp!]
  \setlength{\abovecaptionskip}{0pt}
  \setlength{\belowcaptionskip}{10pt}
  \caption{CPT of agent 6 with $u_6=\mathrm{OR}$ for $\mathcal{G}_s$ of Fig.~\ref{CPT}}
  \centerline{
    \label{cptOR}
    \begin{tabular}{c c c | c c}
      \toprule[1pt]
      $\bar{x}_3$ & $\bar{x}_4$ & $\bar{x}_5$ & $\mathbb{P}^{ \mathrm{OR}}(\bar{x}_6\!=\!0|\bar{x}_3\!,\! \bar{x}_4\!,\! \bar{x}_5)$ & $\mathbb{P}^{ \mathrm{OR}}(\bar{x}_6\!=\!1|\bar{x}_3\!,\! \bar{x}_4\!,\! \bar{x}_5)$ \\
      \midrule[1.25pt]
      0 & 0 & 0 & 1.00 & 0.00 \\
      0 & 0 & 1 & 0.25 & 0.75 \\
      0 & 1 & 0 & 0.50 & 0.50 \\
      0 & 1 & 1 & 0.00 & 1.00 \\
      1 & 0 & 0 & 0.50 & 0.50 \\
      1 & 0 & 1 & 0.00 & 1.00 \\
      1 & 1 & 0 & 0.25 & 0.75 \\
      1 & 1 & 1 & 0.00 & 1.00 \\
      \bottomrule[1pt]
    \end{tabular}}
\end{table}
\begin{table}[htbp!]
  \setlength{\abovecaptionskip}{0pt}
  \setlength{\belowcaptionskip}{10pt}
  \caption{CPT of agent 6 with $u_6=\mathrm{AND}$  for $\mathcal{G}_s$ of Fig.~\ref{CPT}}
  \centerline{
    \label{cptAND}
    \begin{tabular}{c c c | c c}
      \toprule[1pt]
      $\bar{x}_3$ & $\bar{x}_4$ & $\bar{x}_5$ & $\mathbb{P}^{ \mathrm{AND}}(\bar{x}_6\!=\!0|\bar{x}_3\!,\! \bar{x}_4\!,\! \bar{x}_5)$ & $\mathbb{P}^{ \mathrm{AND}}(\bar{x}_6\!=\!1|\bar{x}_3\!,\! \bar{x}_4\!,\! \bar{x}_5)$\\
      \midrule[1.25pt]
      0 & 0 & 0 & 1.00 & 0.00 \\
      0 & 0 & 1 & 0.75 & 0.25 \\
      0 & 1 & 0 & 1.00 & 0.00 \\
      0 & 1 & 1 & 0.50 & 0.50 \\
      1 & 0 & 0 & 1.00 & 0.00 \\
      1 & 0 & 1 & 0.50 & 0.50 \\
      1 & 1 & 0 & 0.75 & 0.25 \\
      1 & 1 & 1 & 0.00 & 1.00 \\
      \bottomrule[1pt]
    \end{tabular}}
\end{table}

We focus on the case when $\proj(\mathcal{G})$ is a DAG. The case when $\proj(\mathcal{G})$ is not a DAG can be handled by by combining the junction tree algorithm and belief propagation. However, a subsequent marginalization within the appropriate junction node maybe required to obtain the desired probability.

\begin{theorem}
    \label{thm:marginal}
    Given a multiplex network $\mathcal{G}$ for which $\proj(\mathcal{G})$ is a DAG, with seed set $S_0$ and sequence of protocols $\mathcal{U}$, the following two probabilities are the same:
    \begin{enumerate}
        \item $\mathbb{P}^{\mathcal{G},\mathcal{U}}_{S_0}(\bar x_i=1)$, the probability that agent $i$ is active at steady state for the multiplex LTM.
        \item $\mathbb{P}^{\mathcal{G}, \mathcal{U}}(\bar x_i=1 | \bar x_j=1, \bar x_l=0,  j\in S_0,  l \notin S_0, N_l=\varnothing)$, 
        the marginal probability of node $i$ in the corresponding Bayesian network of Algorithm \ref{alg:map}, conditioned on observed nodes in the seed set and those not in the seed set that have no out-neighbors. 
    \end{enumerate}
\end{theorem}
\begin{proof}
Since the event that node $i$ is activated is conditioned on the state of node $i$'s out-neighbors in proj($\mathcal{G}$), it is independent of the state of all other nodes. 
Thus, the joint probability of activation of each node $i$ factors into components based on the graphical structure of the Bayesian network constructed in Algorithm \ref{alg:map}. 
So, the probability that agent $i$ is active at steady state for the multiplex LTM (probability in the first statement) can be computed by probabilistic inference on the corresponding Bayesian network (the marginal probability in the second statement). 
The values of the observed nodes in the marginal probability are obtained as follows. Each  $j\in S_0$ is always active, so we observe $\bar x_j=1$.  The state $x_l$ of  $l$ that has no out-neighbors does not change over time. 
If also $l \notin S_0$, then $l$ is always inactive, and  
we observe $\bar x_l = 0$.
\end{proof}

The equivalence in Theorem~\ref{thm:marginal} implies that we can compute multiplex influence spread using probabilistic inference in BNs, which can be solved with BP algorithms.
\begin{corollary}
    \label{cor:is}
    Given a multiplex network $\mathcal{G}$ for which $\proj(\mathcal{G})$ is a DAG, with seed set $S_0$ and sequence of protocols $\mathcal{U}$, multiplex influence spread $\sigma^{\mathcal{G},\mathcal{U}}_{ S_0}$ can be computed as
    \begin{equation*}
         \sum_{i=1}^n \mathbb{P}^{\mathcal{G}, \mathcal{U}}(\bar x_i=1 | \bar x_j=1, \bar x_l=0, j\in S_0,  l \notin S_0, N_l=\varnothing).
    \end{equation*}
\end{corollary}
\begin{proof}
This follows from Theorem \ref{thm:marginal} and Definition \ref{def:mis}.
\end{proof}

Nguyen and Zheng \cite{nguyen2012influence} showed that computing influence spread in a monoplex DAG with the Independent Cascade Model (ICM) is $\#$P-complete. Here, we prove a similar result for computing influence spread for the multiplex LTM.  The implication is that approximating influence spread for the multiplex LTM is the best we can do for large networks.

\begin{theorem}
    Consider a multiplex network $\mathcal{G}$ for which $\proj(\mathcal{G})$ is a DAG, with seed set $S_0$ and sequence of protocols $\mathcal{U}$. Computing $\sigma^{\mathcal{G},\mathcal{U}}_{ S_0}$ for the multiplex LTM is $\#$P-complete. 
\end{theorem}



\begin{proof}
The multiplex LTM problem, i.e., computing $\sigma^{\mathcal{G},\mathcal{U}}_{ S_0}$ for the multiplex LTM, is $\#$P-complete if (i) it is $\#$P-hard and (ii) it is in $\#$P.  We first prove (ii). By Corollary~\ref{cor:is}, every instance of the  multiplex LTM problem can be reduced to a marginalization problem.  Since the marginalization problem is  $\#$P-complete, the multiplex LTM problem is in $\#$P.

We prove (i) by showing that the multiplex LTM problem is a reduction from the ICM problem, which has been shown to be $\#$P-complete (Theorem 1 of \cite{nguyen2012influence}).  The ICM problem refers to computing the influence spread $\sigma_{S_0}^{G,\mathrm{ICM}}$ for the ICM on a monoplex DAG $G=(V,E)$ with seed set $S_0$ and  probability $w_{j, i}$ assigned to each edge $e_{j, i}\in E$.  
   
    Let $m$ be the largest number of out-neighbors over all nodes in $V$. Consider a multiplex network with $m$ layers $G_1,...,G_m$ where $G_k=(V,E^k)$. 
    To define edge sets $E^k$, assign all edges in $G$ and to the multiplex network such that for each node there is at most one outgoing edge in each layer $k$. Let $w_{j, i}$ be the weight of edge from $i$ to $j$ in the multiplex network. 
    Define a set $V'$ with a node $i'\in V'$ for each node $i\in V$. For each $i$ and $k$, compute the sum of weights of $i$'s outgoing edges in layer $k$. If the sum is not 1, create an edge $e_{i,i'}^k\in E'^k$ from $i$ to $i'$ and assign to the edge  a weight that makes the sum equal 1. Let multiplex network $\mathcal{G}'$ have $m$ layers  $G_1',...,G_m'$ where $G_k'=(V \cup V',E^k \cup E'^k)$. Then $\proj(\mathcal{G}')$ is a DAG. Further, every node $i' \in V'$ has no out-neighbors and  $i' \notin S_0$. So in the multiplex LTM $i' \in V'$ remains inactive.
    Let $\mathcal{U}=\{\mathrm{OR}, ...\mathrm{OR} \}$.
    By construction,  $\sigma^{\mathcal{G}',\mathcal{U}}_{ S_0}=\sigma_{S_0}^{G,\mathrm{ICM}}$. 
\end{proof}


\section{Analytical Expressions of Influence Spread} 
We derive analytical expressions for multiplex cascade centrality  for two illustrative classes of the LTM with a two-layer (duplex) network and $N$ agents.  In Section~\ref{sec:duplexrepeat}, each of the layers is the same path network.  This means that each agent has the same neighbors for each sensing modality;  for example,  each agent can see and hear its neighbors. Our results reveal how the cascade is affected when agents distinguish between signals rather than project them.  
If we view the path network as a cycle network with one link missing, then the duplex permutation network we study in Section~\ref{sec:duplexpermutation} has one layer missing the link between agents 1 and $N$  and the other layer missing the link between agents $N-1$ and $N$. 
\subsection{Duplex Repeated Path Network}
\label{sec:duplexrepeat}
Let $\mathcal{G}_R$ be the duplex repeated path network  of Fig.~\ref{fig:reduplex}, which has the monoplex path network $G_{Pa} = \textrm{proj}(\mathcal{G}_R)$ on each of its two layers. 
For any $N$ and any agent $j$, monoplex cascade centrality  $\mathcal{C}_j^{G_{Pa}}$ and multiplex cascade centralities $\mathcal{C}_j^{\mathcal{G}_R, \mathrm{OR}}$ and $\mathcal{C}_j^{\mathcal{G}_R, \mathrm{AND}}$ can be expressed analytically as follows. 

\begin{figure}[htbp]
    \centering
    \includegraphics[width=1.7in]{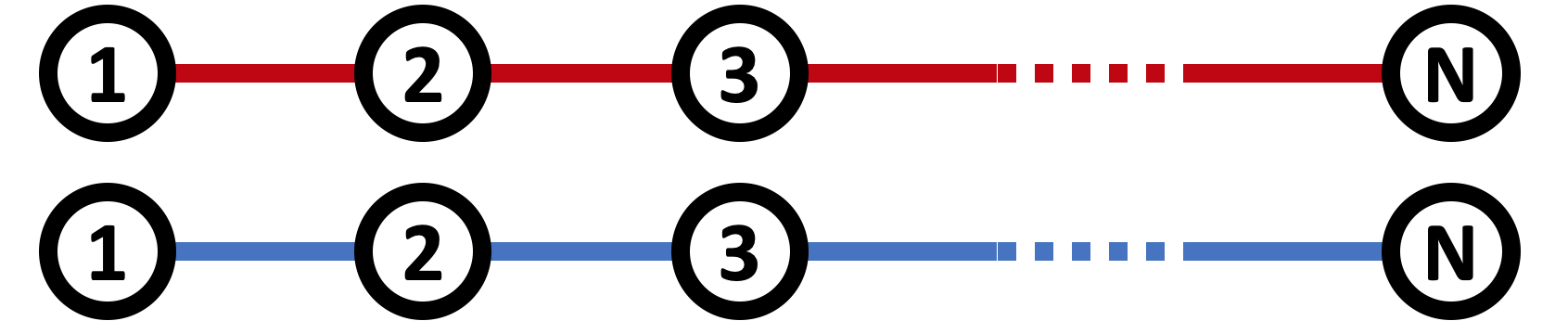}
    \caption{Duplex repeated path network $\mathcal{G}_R$ has path graph $G_{Pa}$ as each layer.}
    \label{fig:reduplex}
\end{figure}

\begin{proposition}[Multiplex cascade centrality for  $\mathcal{G}_R$]
\label{prop:repeated}
    Consider the monoplex path network $G_{Pa}$ and duplex repeated path network $\mathcal{G}_R$ for $N$ agents with $u_i = u \in \{\textrm{OR},\text{AND}\}$. 
    Then 
    \[
    \mathcal{C}_j^{G_{Pa}} = h_j(.5), \;\; \mathcal{C}_j^{\mathcal{G}_R, \mathrm{OR}} =h_j(.75), \;\; \mathcal{C}_j^{\mathcal{G}_R, \mathrm{AND}}=h_j(.25), 
    \]
    \begin{equation*}
        h_j(p_0) \!=\! 
        \begin{cases}
            \sum_{l=0}^{N-2}p_0^l + p_0^{N-2},  \;\;\;j \in \{1,N\} \\
            1 + \sum_{l=0}^{N-3}p_0^l + p_0^{N-3}, \;\;\;j \in \{2,N-1\} \\
            \sum_{l=0}^{j-1}p_0^l \!+\! p_0^{j-1} \!+\! \sum_{l=1}^{N \!-\! j \!-\! 1}p_0^l \!+\! p_0^{N \!-\! j \!-\! 1}, \;\;\; \mathrm{o.w.}
        \end{cases}
    \end{equation*}
Moreover,
    \[
\mathcal{C}_j^{\mathcal{G}_R, \mathrm{OR}} > \mathcal{C}_j^{G_{Pa}} > \mathcal{C}_j^{\mathcal{G}_R, \mathrm{AND}}.
\]
\end{proposition}
\begin{proof}
The result can be derived from Algorithm \ref{alg}. Here, we provide a perspective from probabilistic inference in BNs. Table \ref{tab:cpt_ex} shows the CPT of agent $i \in \{2, \ldots, N-1\}$ for $G_{Pa}$ and for $\mathcal{G}_R$ with $u=\textrm{OR}$ and $u=\textrm{AND}$.
\begin{table}[htbp!]
    \setlength{\abovecaptionskip}{0pt}
    \setlength{\belowcaptionskip}{10pt}
    \caption{CPT of agent $i\in \{2,3,...,N-1\}$. For $G_{Pa}$, $p_0=.5$; for $\mathcal{G}_R$ with $u = \mathrm{OR}$, $p_0=.75$; for $\mathcal{G}_R$ and $u = \mathrm{AND}$, $p_0=.25$.}
    \centerline{
    \begin{tabular}{c c | c c}
        \toprule[1pt]
        $x_{i-1}$ & $x_{i+1}$ & $\mathbb{P}^{\mathcal{G}, u} (x_i\!=\!0| x_{i\!-\!1}, x_{i\!+\!1})$ & $\mathbb{P}^{\mathcal{G}, u} (x_i\!=\!1| x_{i\!-\!1}, x_{i\!+\!1})$ \\
        \midrule[1.25pt]
        0 & 0 & 1 & 0 \\
        0 & 1 & $1-p_0$ & $p_0$ \\
        1 & 0 & $1-p_0$ & $p_0$ \\
        1 & 1 & 0 & 1 \\
        \bottomrule[1pt]
    \end{tabular}}
    \label{tab:cpt_ex}
\end{table}
With one initially active agent $j$, the activity can only spread from an agent closer to $j$ to an agent farther from $j$. Assume $1 < j < i < N$, for all three networks, the probability that agent $i$ is active at steady state for the LTM can be factorized as follows:
\begin{align}
    \label{eqn:ex_fac}
    \mathbb{P}^{\mathcal{G}, u}_j (\bar x_i = 1) = & \; \mathbb{P}^{\mathcal{G}, u} (\bar x_i \!=\! 1 | \bar x_{i-1}\!=\!1, x_{i+1} \!=\! 0) \\ &\times \mathbb{P}^{\mathcal{G}, u} (\bar x_{i-1} \!=\! 1 | \bar x_{i-2}\!=\!1, x_{i} \!=\! 0)\times \cdots \nonumber\\ &\times \mathbb{P}^{\mathcal{G}, u} (\bar x_{j+1} \!=\! 1 | \bar x_{j}\!=\!1, x_{j+2} \!=\! 0) = p_0^{i-j}.\nonumber
\end{align}
The last equality holds since each agent uses the same protocol and so each conditional probability is $p_0$.
The cascade centralities follow by Definitions \ref{def:mis} and \ref{def:mcc}. 
Cases $i,j \in\{ 1, N\}$ 
are calculated similarly.  The inequality follows since $\mathbb{P}_j^{\mathcal{G}_R, \mathrm{OR}}(\bar x_i \!=\! 1) > \mathbb{P}_j^{G_{Pa}}(\bar x_i \!=\! 1) > \mathbb{P}_j^{\mathcal{G}_R, \mathrm{AND}}(\bar x_i \!=\! 1)$.
\end{proof}

Proposition~\ref{prop:repeated} provides a systematic way to evaluate spread for any number of agents $N$ in the multiplex LTM on $\mathcal{G}_R$. The inequality is consistent with the intuition in 
Remark~\ref{rem:spreading}: when agents can distinguish  signals from different sensing modalities and use Protocol OR (AND), they are more (less) easily activated, and the cascade is enhanced (diminished) relative to when agents cannot distinguish signals.

\subsection{Duplex Permutation Networks}
\label{sec:duplexpermutation}
Let $\mathcal{G}_P$ be the duplex permutation network of Fig. \ref{permu}. 
Then $\textrm{proj}(\mathcal{G}_P) = G_C$, the cycle network.
We derive analytical expressions for $\mathbb{P}_j^{\mathcal{G}_P, u} (\bar x_i = 1)$, $u \in \{\textrm{OR},\textrm{AND}\}$, as follows.

\begin{figure}[htbp]
    \centering
    \includegraphics[width=1.5in]{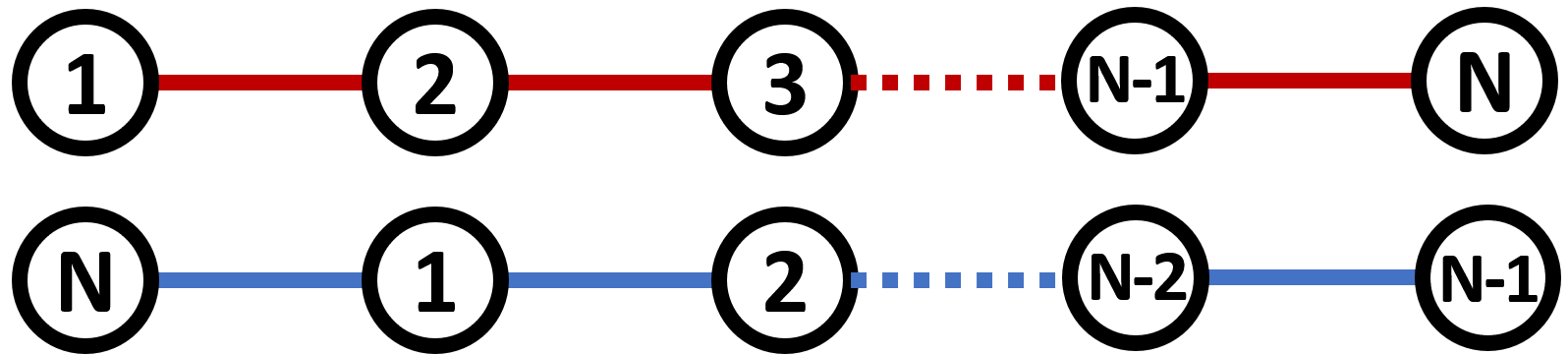}
    \caption{Duplex permutation network $\mathcal{G}_P$.}
    \label{permu}
\end{figure}

\begin{proposition}[Probabilities for multiplex cascade centrality for $\mathcal{G}_P$]
\label{prop:permuation}
    Consider the duplex permutation network $\mathcal{G}_P$ for $N$ agents with $u_i = u \in \{\textrm{OR},\textrm{AND}\}$ and the cyclic network $G_C = \textrm{proj}(\mathcal{G}_P)$. 
    Then
    \begin{align}
        \label{eqn:g_p}
        \mathbb{P}_j^{\mathcal{G}_P, \mathrm{OR}} (\bar x_i \!=\! 1) &\!=\! (.75)^{|i\!-\!j|} \!+\! .5(.75)^{N\!-\!|i\!-\!j|\!-\!3} \!-\! .5(.75)^{N\!-\!5} \nonumber\\ \mathbb{P}_j^{\mathcal{G}_P, \mathrm{AND}}(\bar x_i \!=\! 1) &\!=\! (.25)^{|i-j|} \\
        \mathbb{P}_j^{G_C}(\bar x_i \!=\! 1) &\!=\! (.5)^{|i-j|} + (.5)^{N-|i-j|} \nonumber
    \end{align}
    where $3\leq i \leq N-3$ and $j = 2, ..., i-2, i+2, ..., N-2$.
Moreover,
    \[
\mathcal{C}_j^{\mathcal{G}_P, \mathrm{OR}} > \mathcal{C}_j^{G_{C}} > \mathcal{C}_j^{\mathcal{G}_P, \mathrm{AND}}
\]
  \[
  \mathbb{P}_j^{\mathcal{G}_P, \mathrm{AND}}(\bar x_i \!=\! 1) \!=\! \mathbb{P}_j^{\mathcal{G}_R, \mathrm{AND}}(\bar x_i \!=\! 1) \!=\! (0.25)^{|i-j|} 
  \]
  \[
  \mathbb{P}_j^{\mathcal{G}_P, \mathrm{OR}}(\bar x_i \!=\! 1) \! > \! \mathbb{P}_j^{\mathcal{G}_R, \mathrm{OR}}(\bar x_i \!=\! 1).
  \]
\end{proposition}
\begin{proof}
    The probabilities derive from Algorithm \ref{alg}. The first inequality follows since $\mathbb{P}_j^{\mathcal{G}_P, \mathrm{OR}}(\bar x_i \!=\! 1) > \mathbb{P}_j^{G_C}(\bar x_i \!=\! 1) > \mathbb{P}_j^{\mathcal{G}_P, \mathrm{AND}}(\bar x_i \!=\! 1)$. 
    The rest follows from (\ref{eqn:ex_fac}) and  (\ref{eqn:g_p}).
\end{proof}
When $u_N = \textrm{AND}$, the activity can only spread in $\mathcal{G}_P$ from  $j$ to  $i$ along the path between $j$ and $i$ on $G_C$ that does not contain $N$.  This explains the equality of probabilities for $\mathcal{G}_P$ and $\mathcal{G}_R$. 
When $u_i = u = \textrm{OR}$, the activity can spread in $\mathcal{G}_P$ from $j$ to  $i$ along either path between $j$ and $i$ on $G_C$. This explains the last inequality, i.e., that the cascade is greater in $\mathcal{G}_P$ than in $\mathcal{G}_R$.
As in Proposition~\ref{prop:repeated}, the first inequality in Proposition~\ref{prop:permuation} is consistent with the intuition in Remark~\ref{rem:spreading}.

\section{Heterogeneity in Protocol} 
\subsection{Small Heterogeneous Multiplex Networks}
We compute multiplex cascade centrality to evaluate for the LTM the role of heterogeneity in the tradeoff between sensitivity of the cascade to a real input and robustness of the cascade to a spurious signal.  Knowing that agents that use Protocol OR enhance the cascade and agents that use  AND diminish the cascade, we examine how to leverage heterogeneity in protocol to advantage. Parametrizing the tradeoff by $c$,
we solve as a function of $c$ for the optimal heterogeneous distribution of agents using OR and agents using AND. 

We investigate with the duplex network of Fig.~\ref{fig:fish_network}, which is small enough that we can compute cascade centrality with Algorithm~\ref{alg}.  There are six agents (nodes 1 to 6) and a seventh node that represents an external signal. When node 7 appears in both layers, as in Fig.~\ref{fig:fish_network}, we interpret it as real.  When node 7 appears in only one layer, we interpret it as spurious. We assume that only one agent (e.g., agent 1 in Fig.~\ref{fig:fish_network}) senses node 7, whether or not it is real or spurious, and that it is equally likely to be any of the six agents that sense it. We assume the edge pointing to the signal has a weight of 1.

The graph in layer 1 (red edges) in Fig.~\ref{fig:fish_network} represents a {\em directed sensing} modality, e.g., a team of robots with front and side facing cameras or a school of fish that see poorly to their rear.  The graph in layer 2 (blue edges) represents a {\em proximity sensing} modality, e.g.,  robots that receive local broadcasts 
or fish that detect local movement with their lateral line. 

\begin{figure}[htbp]
  \centering
  \includegraphics[width=1.7in]{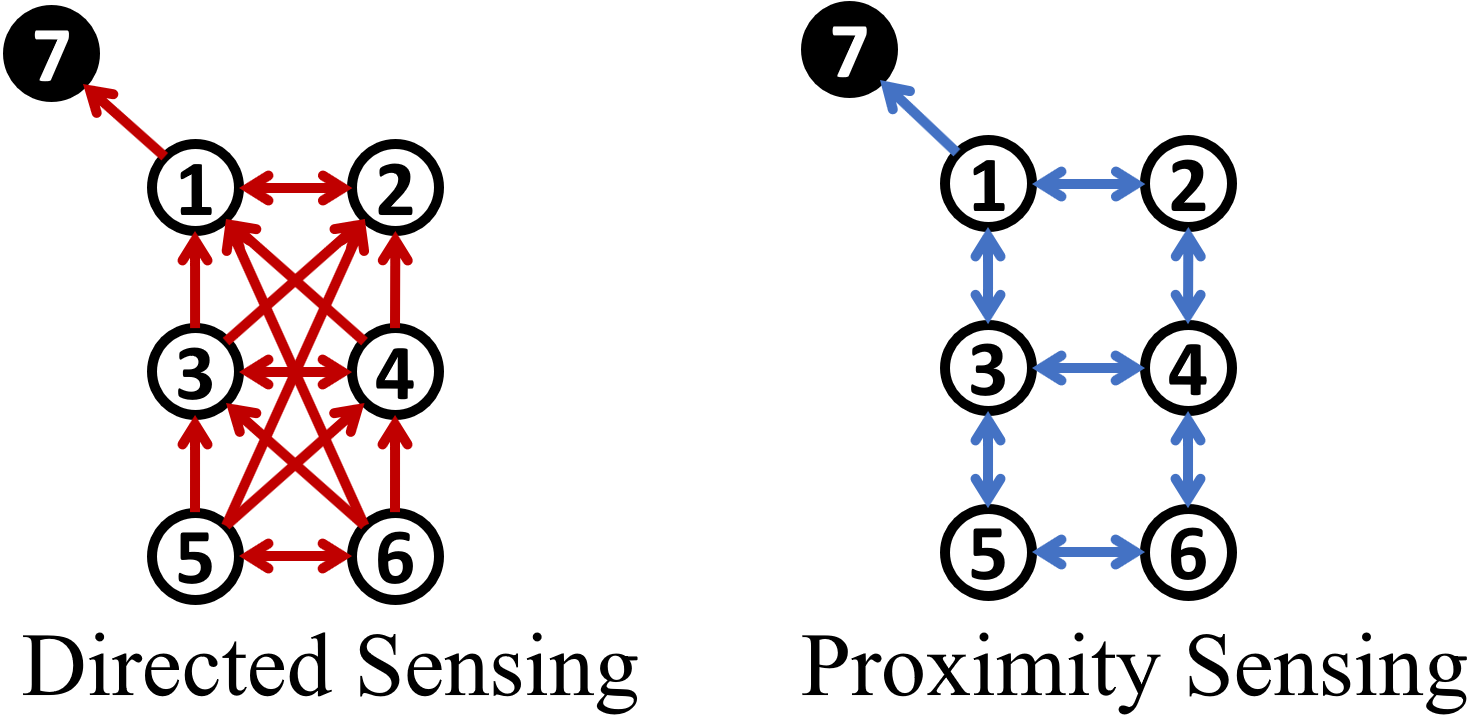}
  \caption{Duplex network with agents 1 to 6, layer 1 (red), and layer 2 (blue). Node 7, the external signal, is real since it appears in both layers.} 
  \label{fig:fish_network}
\end{figure}


As each agent can use either protocol, there are $2^6 = 64$ different possible sequences of protocols $\mathcal{U}$ in total. 
We define the utility function $Q$ as a function of $\mathcal{U}$ and $c \geq 0$ to measure the benefit of cascades that result from real signals less the cost of cascades that result from spurious signals:
\begin{equation}
    Q(\mathcal{U}, c) = \frac{1}{6}\sum_{l=1}^6 \big( \mathcal{C}_7^{\mathcal{G}^l_{\mathrm{real}},\mathcal{U}} - c \frac{1}{2} (\mathcal{C}_7^{\mathcal{G}^l_{\mathrm{spur1}}, \mathcal{U}} + \mathcal{C}_7^{\mathcal{G}^l_{\mathrm{spur2}}, \mathcal{U}})  \big).
\end{equation}
Superscript $l$ indexes the agent sensing node 7. 
Subscripts ``real", ``spur1" and ``spur2" index the networks where node 7 appears in both layers, layer 1 only, and layer 2 only, respectively. Increasing  $c$ increases  cost of response to spurious signals relative to  benefit of response to real signals. 
Given $c$, the optimal sequence of protocols is $\mathcal{U}^c = \mathrm{argmax}_{\mathcal{U}} Q(\mathcal{U}, c)$.  

Fig. \ref{fig:percentage} illustrates $\mathcal{U}^c$ (with symmetry implied) on a plot of the optimal fraction of agents using AND as a function of $c$. A white (gray) circle represents an agent using OR (AND).  When $c$ is 0 or small, responding to real signals dominates and all agents use OR. 
When $c$ is increases towards 3 and beyond, avoiding spurious signals dominates and all agents use AND. 
For $c$ in between, the optimal solution is heterogeneous with more agents using AND as $c$ increases: first agent 1 or 2, then agents 1 and 2, then  \{1,2,5\} or \{1,2,6\}, then \{1,2,3,6\} or \{1,2,4,5\}, and then  \{1,2,3,5,6\} or \{1,2,4,5,6\}.

\begin{figure}[htbp]
  \centering
  \includegraphics[width=0.37\textwidth]{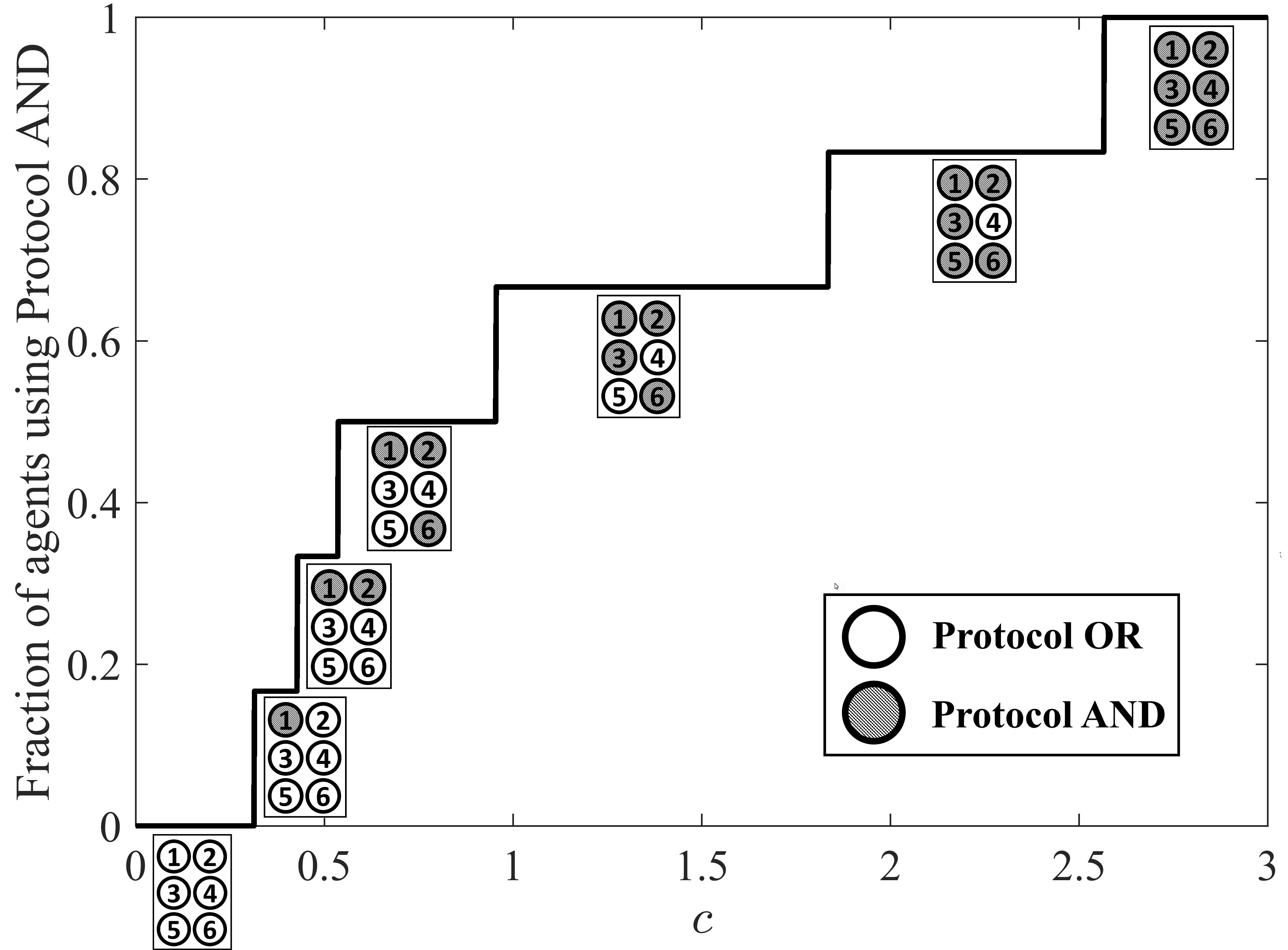}
  \caption{The optimal fraction of agents using Protocol AND with  illustration of optimal solution $\mathcal{U}^c$ as $c$ varies from 0 to 3. Symmetry is implied.} 
  \label{fig:percentage}
\end{figure}

\subsection{Large Heterogeneous Multiplex Networks}

We apply Corollary \ref{cor:is} to study multiplex cascade centrality for a random multiplex network with 20 agents and 
homogeneous and heterogeneous protocols. 
We randomly generate duplex networks, for which the projection networks are DAGs, by fixing a topological order of nodes and assigning edges randomly with probability $p_e$.
A higher probability $p_e$ means agents sense a greater number of the other agents.
We consider homogeneous groups, where $u = \textrm{OR}$ and $u = \textrm{AND}$. 
We also consider heterogeneous groups in which each agent randomly chooses OR or AND with equal likelihood.

We let the root node be the initial active agent and study how the cascade centrality of the root changes as we vary $p_e$ from 0 to 1. For every value of $p_e$, we randomly generate 400 networks. 
Fig.~\ref{fig:trend} shows how the cascade centrality, averaged over the random networks, changes as a function of $p_e$.
\begin{figure}[htbp]
  \centering
  \includegraphics[width=2.0in]{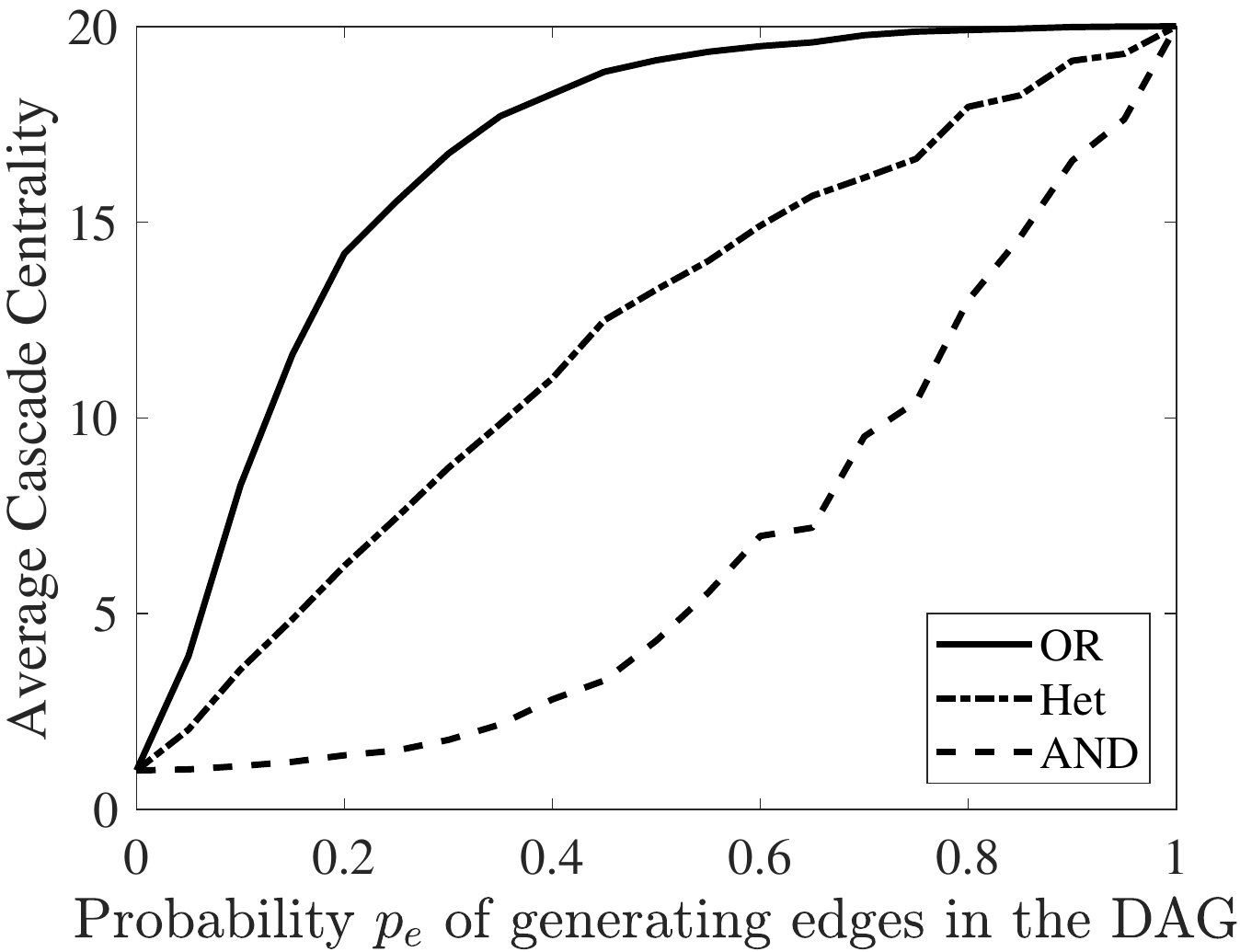}
  \caption{Multiplex cascade centrality of root node, averaged over 400 networks, as a function of probability $p_e$ of edges in the DAG.}
  \label{fig:trend}
\end{figure}

Regardless of the protocol, as $p_e \rightarrow 0$, the DAG become disconnected
and the cascade centrality goes to 1 (only the root node is active at steady state). As $p_e \rightarrow 1$, the root node activates every other agent and the cascade centrality goes to 20. For $p_e$ in between, Fig.~\ref{fig:trend} shows 
that the homogeneous groups with $u=\textrm{OR}$ are most readily activated and cascade size is sensitive  in the range $p_e \in (0,.5)$. Homogeneous groups with $u=\textrm{AND}$ are least readily activated and cascade size is sensitive  in the range $p_e \in (.5,1)$. 

\section{Conclusion}
We have extended the LTM to multiplex networks where agents use different protocols that distinguish signals from multiple sensing modalities. We have derived algorithms to compute influence spread accurately using the multiplex LEM and approximately using probabilistic inference. We have shown how multiple sensing modalities affect spread and how heterogeneity trades off sensitivity and robustness of spread.

\bibliographystyle{IEEEtran}
\bibliography{IEEEabrv,myref}

%









\end{document}